\documentclass[10pt]{article}

\usepackage{enumerate,xspace}
\usepackage{amsmath,amssymb,wasysym}
\usepackage[all]{xy}
\usepackage{proof}
\usepackage[svgnames]{xcolor}
\usepackage{pict2e} 
\usepackage{tikz}
\usepackage{stmaryrd} 
\usepackage{mathtools}
\usepackage{latexsym}
\usepackage{dsfont}
\usepackage{multicol}
\usepackage{cmll}
\usepackage{multirow}
\usepackage{longtable}
\usepackage[sort,nocompress]{cite} 

\usepackage{stackrel}

\usepackage{lscape}
\usepackage{array}

\usepackage{hyperref} 
\hypersetup{
    colorlinks,
    citecolor=red,
    filecolor=red,
    linkcolor=blue,
    urlcolor=red
}

\usepackage{shuffle}

\newtheorem{observation}{Remark}[section]
\newtheorem{lemma}[observation]{Lemma}  
\newtheorem{theorem}[observation]{Theorem}
\newtheorem{definition}[observation]{Definition}
\newtheorem{example}[observation]{Example}
\newtheorem{remark}[observation]{Remark}

\newtheorem{proposition}[observation]{Proposition} 
\newtheorem{corollary}[observation]{Corollary}

\newcommand{\wand}{\ensuremath{
  \mathrel{\vbox{\offinterlineskip\ialign{
    \hfil##\hfil\cr
    $\star$\cr
    \noalign{\kern-1ex}
    $\vert$\cr
}}}}}

\makeatletter

\newdimen\w@dth

\def\setw@dth#1#2{\setbox\z@\hbox{\scriptsize $#1$}\w@dth=\wd\z@
\setbox\@ne\hbox{\scriptsize $#2$}\ifnum\w@dth<\wd\@ne \w@dth=\wd\@ne \fi
\advance\w@dth by 1.2em}

\def\t@^#1_#2{\allowbreak\def\n@one{#1}\def\n@two{#2}\mathrel
{\setw@dth{#1}{#2}
\mathop{\hbox to \w@dth{\rightarrowfill}}\limits
\ifx\n@one\empty\else ^{\box\z@}\fi
\ifx\n@two\empty\else _{\box\@ne}\fi}}
\def\t@@^#1{\@ifnextchar_ {\t@^{#1}}{\t@^{#1}_{}}}

\def\t@left^#1_#2{\def\n@one{#1}\def\n@two{#2}\mathrel{\setw@dth{#1}{#2}
\mathop{\hbox to \w@dth{\leftarrowfill}}\limits
\ifx\n@one\empty\else ^{\box\z@}\fi
\ifx\n@two\empty\else _{\box\@ne}\fi}}
\def\t@@left^#1{\@ifnextchar_ {\t@left^{#1}}{\t@left^{#1}_{}}}

\def\two@^#1_#2{\def\n@one{#1}\def\n@two{#2}\mathrel{\setw@dth{#1}{#2}
\mathop{\vcenter{\hbox to \w@dth{\rightarrowfill}\kern-1.7ex
                 \hbox to \w@dth{\rightarrowfill}}%
       }\limits
\ifx\n@one\empty\else ^{\box\z@}\fi
\ifx\n@two\empty\else _{\box\@ne}\fi}}
\def\tw@@^#1{\@ifnextchar_ {\two@^{#1}}{\two@^{#1}_{}}}

\def\tofr@^#1_#2{\def\n@one{#1}\def\n@two{#2}\mathrel{\setw@dth{#1}{#2}
\mathop{\vcenter{\hbox to \w@dth{\rightarrowfill}\kern-1.7ex
                 \hbox to \w@dth{\leftarrowfill}}%
       }\limits
\ifx\n@one\empty\else ^{\box\z@}\fi
\ifx\n@two\empty\else _{\box\@ne}\fi}}
\def\t@fr@^#1{\@ifnextchar_ {\tofr@^{#1}}{\tofr@^{#1}_{}}}

\newdimen\W@dth
\def\setW@dth#1#2{\setbox\z@\hbox{$#1$}\W@dth=\wd\z@
\setbox\@ne\hbox{$#2$}\ifnum\W@dth<\wd\@ne \W@dth=\wd\@ne \fi
\advance\W@dth by 1.2em}

\def\T@^#1_#2{\allowbreak\def\N@one{#1}\def\N@two{#2}\mathrel
{\setW@dth{#1}{#2}
\mathop{\hbox to \W@dth{\rightarrowfill}}\limits
\ifx\N@one\empty\else ^{\box\z@}\fi
\ifx\N@two\empty\else _{\box\@ne}\fi}}
\def\T@@^#1{\@ifnextchar_ {\T@^{#1}}{\T@^{#1}_{}}}

\def\T@left^#1_#2{\def\N@one{#1}\def\N@two{#2}\mathrel{\setW@dth{#1}{#2}
\mathop{\hbox to \W@dth{\leftarrowfill}}\limits
\ifx\N@one\empty\else ^{\box\z@}\fi
\ifx\N@two\empty\else _{\box\@ne}\fi}}
\def\T@@left^#1{\@ifnextchar_ {\T@left^{#1}}{\T@left^{#1}_{}}}

\def\Tofr@^#1_#2{\def\N@one{#1}\def\N@two{#2}\mathrel{\setW@dth{#1}{#2}
\mathop{\vcenter{\hbox to \W@dth{\rightarrowfill}\kern-1.7ex
                 \hbox to \W@dth{\leftarrowfill}}%
       }\limits
\ifx\N@one\empty\else ^{\box\z@}\fi
\ifx\N@two\empty\else _{\box\@ne}\fi}}
\def\T@fr@^#1{\@ifnextchar_ {\Tofr@^{#1}}{\Tofr@^{#1}_{}}}

\def\Two@^#1_#2{\def\N@one{#1}\def\N@two{#2}\mathrel{\setW@dth{#1}{#2}
\mathop{\vcenter{\hbox to \W@dth{\rightarrowfill}\kern-1.7ex
                 \hbox to \W@dth{\rightarrowfill}}%
       }\limits
\ifx\N@one\empty\else ^{\box\z@}\fi
\ifx\N@two\empty\else _{\box\@ne}\fi}}
\def\Tw@@^#1{\@ifnextchar_ {\Two@^{#1}}{\Two@^{#1}_{}}}

\def\to{\@ifnextchar^ {\t@@}{\t@@^{}}}
\def\from{\@ifnextchar^ {\t@@left}{\t@@left^{}}}
\def\tofro{\@ifnextchar^ {\t@fr@}{\t@fr@^{}}}
\def\To{\@ifnextchar^ {\T@@}{\T@@^{}}}
\def\From{\@ifnextchar^ {\T@@left}{\T@@left^{}}}
\def\Two{\@ifnextchar^ {\Tw@@}{\Tw@@^{}}}
\def\Tofro{\@ifnextchar^ {\T@fr@}{\T@fr@^{}}}

\makeatother

\title{Fundamental Theorems of Calculus and Zinbiel Algebras}
\author{Jean-Simon Pacaud Lemay}

\begin{document}
\allowdisplaybreaks

\maketitle

 \begin{abstract} Derivations are linear operators which satisfy the Leibniz rule, while integrations are linear operators which satisfy the Rota-Baxter rule. In this paper, we introduce the notion of an FTC-pair, which consists of an algebra and module with a derivation and integration between them which together satisfy analogues of the two Fundamental Theorems of Calculus. In the special case of when an algebra is seen as a module over itself, we show that this sort of FTC-pair is precisely the same thing as an integro-differential algebra. We provide various constructions of FTC-pairs, as well as numerous examples of FTC-pairs including some based on polynomials, smooth functions, Hurwitz series, and shuffle algebras. Moreover, it is well known that integrations are closely related to Zinbiel algebras. The main result of this paper is that the category of FTC-pairs is equivalent to the category of Zinbiel algebras. \end{abstract}

 \noindent \small \textbf{Acknowledgements.} The author would like to thank Richard Blute for many useful discussions and their support of this project. The author would also like to thank the attendees of the Atlantic Category Theory Seminar (in particular Geoff Crutwell), the Australian Category Seminar (in particular Richard Garner, Ross Street, Steve Lack, and John Powers), and the Australian Algebra Conference for listening to this story, their questions, and the encouragement to write up this paper. This material is based upon work supported by the AFOSR under award number FA9550-24-1-0008. The author is also funded by an ARC DECRA (DE230100303). 

 \tableofcontents


\section{Introduction}

The two main branches of calculus are differential calculus and integral calculus, which are related via the two Fundamental Theorems of Calculus. The use of algebraic methods to study differential operators and integral operators is an important field of research and has been quite successful. 

Famously, the algebraic generalization of the differential operator is a \textbf{derivation} \cite{matsumura1970commutative} which is a linear operator from an algebra to a module which satisfies an analogue of the \textit{Leibniz rule} from calculus (Def \ref{def:der-int}): 
\[(fg)^\prime(x)= f^\prime(x) g(x) + f(x) g^\prime(x)\] 
Derivations are, of course, very well-known and studied, and they have important applications in algebra, algebraic geometry, differential geometry, and various other areas. In the special case that the algebra is seen as a module over itself, an algebra equipped with a derivation is called a \textbf{differential algebra} \cite{kaplansky1957introduction}. 

On the other hand, the algebraic generalization of the integral operator is called an \textbf{integration} \cite{bagnol2016shuffle} (also called a generalized Rota-Baxter operator \cite{uchino2008quantum} or a relative Rota-Baxter operator of weight $0$ \cite{bai2013relative}) and is a linear operator which instead goes from a module to an algebra which satisfies an analogue of an integral only version of the integration by parts rule called the \textit{Rota-Baxter rule} (Def \ref{def:der-int}): 
\[\left( \int\limits^x_a f(t) \mathsf{d}t \right) \left( \int\limits^x_a g(t) \mathsf{d}t \right) =  \int\limits^x_a f(t)  \left(\int\limits^t_a g(u) \mathsf{d}u \right) \mathsf{d}t + \int\limits^x_a \left(\int\limits^t_a f(u) \mathsf{d}u \right) g(t) \mathsf{d}t\] 
Again, in the special case of an algebra seen as a module over itself, an algebra equipped with an integration is called a \textbf{Rota-Baxter algebra} \cite{guo2012introduction}.

It is well established that integrations are closely related to Zinbiel algebras and dendriform algebras. \textbf{Zinbiel algebras} (Def \ref{def:Zinbiel-algebra}) were introduced by Loday in \cite{loday1995cup} (under the name dual Leibniz algebras), and the Zinbiel algebra operad is the Koszul dual of the Leibniz algebra operad (hence the name Zinbiel, which is Leibniz backwards). Zinbiel algebras have found applications throughout a variety of fields -- see \cite{alvarez2022algebraic} for a great list of interesting applications of Zinbiel algebras. Dendriform algebras were also introduced by Loday \cite{Loday2001} and can be understood as the ``non-commutative generalization'' of a Zinbiel algebra -- so, in particular, every Zinbiel algebra is a dendriform algebra \cite[Lemma 7.2]{Loday2001} (in fact Zinbiel algebras are also sometimes called commutative dendriform algebras \cite[Def 5.2.1.(a)]{guo2012introduction}). 

In the non-commutative setting, every integration induces a dendriform algebra \cite[Prop 2.10]{uchino2008quantum}, and conversely every dendriform algebra gives rise to an integration \cite[Sec 2.2]{uchino2008quantum}. Similarly in the commutative setting, every integration induces a Zinbiel algebra, and conversely, from every Zinbiel algebra we can construct an integration. However, these constructions are not inverses of each other since, in particular, not every integration arises from a dendriform/Zinbiel algebra. That said, Uchino in \cite[Sec 2.2]{uchino2008quantum} still showed that there is an adjunction between a certain category of integrations and the category of dendriform algebras. Afterwards, Bai, Guo, and Ni in \cite{bai2013relative} showed that dendriform algebras correspond precisely to \emph{invertible} integrations. It is straightforward to see that these results also hold for the commutative case, so we have an adjunction between a certain category of integrations and the category of Zinbiel algebras, and that Zinbiel algebras correspond precisely to integrations which are isomorphisms. 

While these results clearly show an already deep connection between Zinbiel algebras and integrations, there are two things we wish to address. The first is that the main results in \cite{bai2013relative} are expressed in terms of bijection correspondences between equivalence classes of integrations and dendriform/Zinbiel algebras, and not as an equivalence of categories. As such, it would be nice to explicitly have an equivalence between the category of Zinbiel algebras and a certain category of integrations. The second is that the requirement of working with invertible integrations is somewhat strong and excludes many natural examples, including \emph{unital} algebras. It turns out, that if one wishes to add units back into the picture and figure out what kind of integrations Zinbiel algebras correspond to in the unital setting, one must add derivations to the story! 

In calculus, the differential operator and the integral operator are related via the two \textit{Fundamental Theorem of Calculus}. Thus, it is natural to study the algebraic analogues of this phenomenon. In light of this, the main novel concept that we introduce and study in this paper is that of an \textbf{FTC-pair} (Def \ref{def:FTCpair}), which consists of a derivation and an integration which together satisfy analogues of the two Fundamental Theorems of Calculus. It turns out that in the same way that a derivation is a ``module generalization'' of a differential algebra and that an integration is a ``module generalization'' of a Rota-Baxter algebra, an FTC-pair is a ``module generalization'' of an \textbf{integro-differential algebra} \cite{guo2014integro}. That said, while an FTC-pair is axiomatized via the Fundamental Theorems of Calculus, an integro-differential algebra is instead axiomatized in terms of an equation called the \textit{hybrid Rota-Baxter rule}. However, we show that assuming the two Fundamental Theorems of Calculus is equivalent to assuming one of them and the hybrid Rota-Baxter rule (Prop \ref{prop:FTC-integro}). This implies that every integro-differential algebra is indeed an FTC-pair in the case that the algebra is viewed as a module over itself. Moreover, it is also possible to construct FTC-pairs from both integrations (Lemma \ref{lem:int-FTC}) and derivations (Lemma \ref{lem:der-FTC}). 

The main result of this paper (Thm \ref{thm:FTC-Zin}) is that the category of FTC-pairs (Def \ref{def:FTC-category}) is equivalent to a larger category of Zinbiel algebras (Def \ref{def:Zinbiel-bigcat}). From this, we then obtain that Zinbiel algebras correspond precisely to \textit{augmented} FTC-pairs (Def \ref{def:aFTC}), or in other words, that the category of Zinbiel algebras over a fixed based ring is equivalent to the category of augmented FTC-pairs (Cor \ref{cor:aFTC-Zin}). This latter result makes sense since the two Fundamental Theorems of Calculus amount to saying that the integration and the derivation are inverses of each other up to constants. Thus, this result still falls in line with the results of \cite{bai2013relative}. Nevertheless, FTC-pairs provide a new perspective on Zinbiel algebras and, in particular, provide a new connection between derivations and Zinbiel algebras (which, to the best of the author's knowledge, has not been previously done). 

It is worth mentioning that while in this paper we work in the commutative case, it is straightforward to see that the results of this paper can also work in the non-commutative case. So, we can easily generalize the main results to obtain a correspondence between dendriform algebras and the non-commutative version of FTC-pairs as well. Our choice to work in the commutative setting and with Zinbiel algebras comes from the fact that the inspiration for the story of this paper arose from working with \textit{differential categories}. 

We conclude this introduction by discussing the original motivation for the results of this paper. Another approach for studying differentiation is that of the theory of differential categories \cite{Blute2019}, which uses category theory to study the foundation of differential calculus. Briefly, a differential category is a symmetric monoidal category with a (co)monad which comes equipped with a \textbf{deriving transformation}, which is a generalization of a derivation for each (co)free (co)algebra of the monad. Differential categories have been quite successful, having found many applications in computer science and have also been able to generalize various notions related to differentiation such as derivations, differential algebras, de Rham cohomology, etc. 

One can also study generalizations of integrations and the Fundamental Theorems of Calculus in a differential category with \textit{antiderivatives} \cite{cockett_lemay_2018}. As such, in future work, it will be interesting to study generalized versions of FTC-pairs in a differential category as well. This is of particular interest since it turns out that for certain well-behaved (and important) differential categories, having antiderivatives is completely determined by an FTC-pair (or equivalently integro-differential algebra) structure on the (co)free (co)algebra over the monoidal unit \cite{lemay2020convenient}. Of course, said result was not stated in this way, and the connection was only made later. We conjecture that it is possible to strengthen this result and obtain a connection between having antiderivatives and having that every (co)free (co)algebra of the (co)monad belongs to an FTC-pair. This potentially opens the door to re-expressing antiderivatives in differential categories via Zinbiel algebras (since differential categories always assume a commutative setting by definition). 

Thus, this paper establishes the groundwork for the use of FTC-pairs/integro-differential algebras and Zinbiel algebras in future work on integration and the Fundamental Theorems of Calculus in differential categories. 

\section{FTC-Pairs}

In this section, we introduce the notion of an \emph{FTC-pair} (where FTC is shorthand for Fundamental Theorems of Calculus), which is the main novel concept of study of this paper. Briefly, an FTC-pair consists of a derivation and integration which together satisfy analogues of the two Fundamental Theorems of Calculus. We will explain how, in the same way that a derivation (resp. an integration) is a ``module generalization'' of a differential algebra (resp. a Rota-Baxter algebra), an FTC pair is a ``module generalization'' of an integro-differential algebra \cite{guo2014integro}. 

Let us begin by reviewing the definitions of derivations and integrations. As mentioned above, integrations have also been called generalized Rota-Baxter operators \cite{uchino2008quantum} or relative Rota-Baxter operators of weight $0$ \cite{bai2013relative} for slightly more general settings. Here, we follow the terminology used in \cite{bagnol2016shuffle} and use the term integration since we are relating the concept to derivations and the Fundamental Theorems of Calculus.  

\begin{definition}\label{def:der-int} Let $k$ be a commutative ring, $A$ a commutative (unital and associative) $k$-algebra, and $M$ an (left) $A$-module. 
\begin{enumerate}[{\em (i)}] 
\item A \textbf{derivation} \cite[Def 26.A]{matsumura1970commutative} is a $k$-linear map $\mathsf{D}: A \to M$ which satisfies the \textbf{Leibniz rule}, that is, the following equality holds for  all $a,b \in A$:
\begin{equation}\label{leibniz}\begin{gathered} 
\mathsf{D}(ab) = a \mathsf{D}(b) + b \mathsf{D}(a) 
 \end{gathered}\end{equation}
 \item An \textbf{integration} \cite[Def 6.2]{bagnol2016shuffle} is a $k$-linear map $\mathsf{P}: M \to A$ which satisfies the \textbf{Rota-Baxter rule}, that is, the following equality holds for all $m,n \in M$:  
\begin{equation}\label{Rota-Baxter}\begin{gathered} 
\mathsf{P}(m) \mathsf{P}(n) = \mathsf{P}\left(\mathsf{P}(m)n \right) + \mathsf{P}\left(\mathsf{P}(n) m \right) 
 \end{gathered}\end{equation}
\end{enumerate}
\end{definition}

In the special case that the algebra is seen as a module over itself, an algebra $A$ equipped with a derivation ${\mathsf{D}: A \to A}$ is called a \textbf{differential algebra} (of weight $0$) \cite{kaplansky1957introduction}, while an algebra $A$ equipped with an integration ${\mathsf{P}: A \to A}$ is called a \textbf{Rota-Baxter algebra} (of weight $0$) \cite{guo2012introduction}. 

As mentioned above, an FTC-pair consists of a derivation and an integration which are compatible via analogues of the two Fundamental Theorems of Calculus. The first Fundamental Theorem of Calculus essentially states that the integral is an antiderivative, or in other words, that the derivative of the integral gives us back the starting function: so if $g(x) = \int\limits^u_a f(t) \mathsf{d}t$ then $g^\prime(x) = f(x)$. As such, for a derivation and an integration, the analogue of the first Fundamental Theorem of Calculus is saying that the integration is a section of the derivation (or equivalently, that the derivation is a retraction of the integration). On the other hand, the second Fundamental Theorem of Calculus states that the integral of a derivative is equal to the difference of the starting function evaluated at the bounds of the integral: $\int \limits^x_a f^\prime(t) \mathsf{d}t = f(x) - f(a)$. So, for a derivation and an integration, the analogue of the second Fundamental Theorem of Calculus is that the derivation is almost a section of the integration, but up to a constant. In this case, by a constant we mean an element in the kernel of the derivation. We also ask that this constant be compatible with the multiplication.  

\begin{definition}\label{def:FTCpair}  Let $k$ be a commutative ring, $A$ a commutative $k$-algebra, and $M$ an $A$-module. 
\begin{enumerate}[{\em (i)}] 
\item A derivation $\mathsf{D}: A \to M$ and an integration $\mathsf{P}: M \to A$ satisfy the \textbf{First Fundamental Theorem of Calculus (FTC1)} \cite[Def 6.4]{bagnol2016shuffle} if $\mathsf{D} \circ \mathsf{P} = \mathsf{id}_M$, that is, the following equality holds for all $m \in M$:  
\begin{equation}\label{FTC1}\begin{gathered} 
\mathsf{D}\left( \mathsf{P}(m) \right) = m  
 \end{gathered}\end{equation}
 \item  A derivation $\mathsf{D}: A \to M$ and an integration $\mathsf{P}: M \to A$ satisfy the \textbf{Second Fundamental Theorem of Calculus (FTC2)} if for every $a\in A$, there exists a \textbf{$\mathsf{D}$-constant}, that is, a $c_a \in \mathsf{ker}\left( \mathsf{D} \right)$, such that the following equality holds for all $a \in A$: 
\begin{equation}\label{FTC2}\begin{gathered} 
\mathsf{P}\left( \mathsf{D}(a) \right) = a - c_a
 \end{gathered}\end{equation}
 and also that this $\mathsf{D}$-constant is compatible with multiplication, that is, $c_{ab} = c_a c_b$. 
\end{enumerate}
An \textbf{FTC-pair} is a pair $(\mathsf{D}: A \to M, \mathsf{P}: M \to A)$, which we will write as $\xymatrix{ A \ar@/^/[r]^-{\mathsf{D}}  &  \ar@/^/[l]^-{\mathsf{P}}  M
  }$, consisting of a derivation $\mathsf{D}: A \to M$ and an integration $\mathsf{P}: M \to A$ which satisfy both FTC1 and FTC2.
\end{definition}

Examples of FTC-pairs can be found below in Ex \ref{ex:FTC-pairs}. In Ex \ref{ex:FTC1/2-sept}, we also give separating examples which show that FTC1 and FTC2 are indeed independent. 

In the special case of $A=M$, an algebra $A$ equipped with a derivation ${\mathsf{D}: A \to A}$ and an integration $\mathsf{P}: A \to A$ which satisfy FTC1 is called a \textbf{differential Rota-Baxter algebra} \cite{guo2008differential}. On the other hand, an algebra equipped with a derivation and an integration which satisfy FTC2 has, to the best of our knowledge, not been defined. Indeed, in algebra literature, FTC1 seems to be the axiom of interest, while in differential category literature, FTC2 is the more prevalent axiom \cite{cockett_lemay_2018, lemay2020convenient}. That said, an algebra $A$ equipped with a derivation ${\mathsf{D}: A \to A}$ and an integration ${\mathsf{P}: A \to A}$ which satisfy both FTC1 and FTC2 is precisely an \textbf{integro-differential algebra} \cite{guo2014integro}. However, the definition of an integro-differential algebra is not given in terms of FTC1 and FTC2, but instead in terms of FTC1 and an axiom called the \textit{hybrid Rota-Baxter rule}. Thus, some explanation is needed to show that an FTC-pair of the form $\xymatrix{ A \ar@/^/[r]^-{\mathsf{D}}  &  \ar@/^/[l]^-{\mathsf{P}}  A
  }$ is precisely an integro-differential algebra. Let us first show that FTC2 (without assuming FTC1) can be equivalently given in terms of the hybrid Rota-Baxter rule or in terms of an algebra morphism. 

\begin{proposition}\label{prop:FTC2} Let $\mathsf{D}: A \to M$ be a derivation and $\mathsf{P}: M \to A$ be an integration. Then the following are equivalent: 
\begin{enumerate}[{\em (i)}]
\item $\mathsf{D}$ and $\mathsf{P}$ satisfy FTC2; 
\item The map $\mathsf{E}: A \to A$, defined as $\mathsf{E}(a) \colon \!\!\!= a - \mathsf{P}\left( \mathsf{D}(a) \right)$, is an idempotent $k$-algebra morphism such that $\mathsf{D} \circ \mathsf{E} =0$; 
\item $\mathsf{D}$ and $\mathsf{P}$ satisfy that $\mathsf{D} \circ \mathsf{P} \circ \mathsf{D} = \mathsf{D}$, and also the \textbf{hybrid Rota-Baxter rule} \cite[Def 2.1]{guo2014integro}, that is, the following equality holds: 
\begin{equation}\label{hybird-Rota-Baxter}\begin{gathered} 
\mathsf{P}\left( \mathsf{D}(a) \right)\mathsf{P}\left( \mathsf{D}(b) \right) + \mathsf{P}\left( \mathsf{D}(ab) \right) = a ~\mathsf{P}\left( \mathsf{D}(b) \right) + b~\mathsf{P}\left( \mathsf{D}(a) \right)
 \end{gathered}\end{equation}
for all $a,b \in A$. 
\end{enumerate}
\end{proposition}
\begin{proof} For $(i) \Rightarrow (ii)$, first note that by construction, $\mathsf{E}$ is indeed a $k$-linear morphism. Moreover, by definition we have that $\mathsf{P}\left( \mathsf{D}(a) \right) = a - \mathsf{E}(a)$. However by FTC2, we also have that $\mathsf{P}\left( \mathsf{D}(a) \right) = a -c_a$, where $c_a \in \mathsf{ker}\left( \mathsf{D} \right)$. So we get that $\mathsf{E}(a) = c_a$, and thus $\mathsf{D} \circ \mathsf{E} =0$. Now, to explain why $\mathsf{E}$ preserves the multiplication and unit. First, recall that for any derivation, the Leibniz rule implies that $\mathsf{D}(1) =0$. Then by FTC2 we get that $0 = 1 - \mathsf{E}(1)$, and so $\mathsf{E}(1)=1$. Next, by the FTC2 assumption, we have that $\mathsf{E}(ab) = c_{ab} = c_a c_b = \mathsf{E}(a) \mathsf{E}(b)$, so $\mathsf{E}$ preserves the multiplication. So we indeed have that $\mathsf{E}$ is a $k$-algebra morphism. Lastly, we need to show that $\mathsf{E}$ is idempotent. Since $\mathsf{D}\left( \mathsf{E}(a) \right) = 0$, by FTC2 we get that $0 = \mathsf{E}(a) - \mathsf{E}\left( \mathsf{E}(a) \right)$, and thus $\mathsf{E}\left( \mathsf{E}(a) \right)= \mathsf{E}(a)$. 

For $(ii) \Rightarrow (iii)$, we first show the hybrid Rota-Baxter rule. Since $\mathsf{E}$ is a $k$-algebra morphism, it in particular preserves the multiplication, so $\mathsf{E}(ab) = \mathsf{E}(a)\mathsf{E}(b)$, or concretely $ab - \mathsf{P}\left( \mathsf{D}(ab) \right)= \left( a -\mathsf{P}\left( \mathsf{D}(a) \right) \right)\left( b -\mathsf{P}\left( \mathsf{D}(b) \right) \right)$. Expanding out the left-hand side gives us that:
\[ ab - \mathsf{P}\left( \mathsf{D}(ab) \right) = ab - b~\mathsf{P}\left( \mathsf{D}(a) \right) - a ~\mathsf{P}\left( \mathsf{D}(b) \right) + \mathsf{P}\left( \mathsf{D}(a) \right)\mathsf{P}\left( \mathsf{D}(b) \right) \]
Subtracting $ab$ from both sides and rearranging, we get that: 
\[ \mathsf{P}\left( \mathsf{D}(a) \right)\mathsf{P}\left( \mathsf{D}(b) \right) + \mathsf{P}\left( \mathsf{D}(ab) \right) = a ~\mathsf{P}\left( \mathsf{D}(b) \right) + b~\mathsf{P}\left( \mathsf{D}(a) \right) \]
as desired. For the other identity, by the assumption that $\mathsf{D} \circ \mathsf{E} =0$, we get that $0 = \mathsf{D}(a) - \mathsf{D}\left( \mathsf{P}\left( \mathsf{D}(a) \right) \right)$, and thus $\mathsf{D}\left( \mathsf{P}\left( \mathsf{D}(a) \right) \right)= \mathsf{D}(a)$ as desired. 

For $(iii) \Rightarrow (i)$, for each $a \in A$, define $c_a \in A$ as $c_a \colon \!\!\!= a - \mathsf{P}\left( \mathsf{D}(a) \right)$. By construction, we clearly have that $\mathsf{P}\left( \mathsf{D}(a) \right) = a - c_a$. We must also explain why $c_a$ is a $\mathsf{D}$-constant. By the assumption that $\mathsf{D} \circ \mathsf{P} \circ \mathsf{D} = \mathsf{D}$, we get that $\mathsf{D}(c_a) = \mathsf{D}(a) - \mathsf{D}\left( \mathsf{P}\left( \mathsf{D}(a) \right) \right) =0$, so $c_a \in \mathsf{ker}\left( \mathsf{D} \right)$ as desired. Now, using the hybrid Rota-Baxter rule, we can show that $c_{ab}$ is equal to $c_a c_b$. So we compute: 
\begin{align*}
c_{ab} = ab - \mathsf{P}\left( \mathsf{D}(ab) \right) = ab - b~\mathsf{P}\left( \mathsf{D}(a) \right) - a ~\mathsf{P}\left( \mathsf{D}(b) \right) + \mathsf{P}\left( \mathsf{D}(a) \right)\mathsf{P}\left( \mathsf{D}(b) \right) =  \left( a -\mathsf{P}\left( \mathsf{D}(a) \right) \right)\left( b -\mathsf{P}\left( \mathsf{D}(b) \right) \right) = c_a c_b
\end{align*}
So we conclude that $\mathsf{D}$ and $\mathsf{P}$ satisfy FTC2. 
\end{proof}

We now extend the above proposition by assuming FTC1. Moreover, in the definition of an integro-differential algebra \cite[Def 2.1]{guo2014integro}, it is not necessary to assume that the ``integro'' part is an integration, since from the hybrid Rota-Baxter rule we can prove the Rota-Baxter rule. 

\begin{proposition}\label{prop:FTC-integro} Let $\mathsf{D}: A \to M$ be a derivation and $\mathsf{P}: M \to A$ a $k$-linear map. Then the following are equivalent: 
\begin{enumerate}[{\em (i)}]
\item $\xymatrix{ A \ar@/^/[r]^-{\mathsf{D}}  &  \ar@/^/[l]^-{\mathsf{P}}  M
  }$ is an FTC-pair (so in particular, $\mathsf{P}$ is an integration);
\item $\mathsf{D}$ and $\mathsf{P}$ satisfy $\mathsf{D} \circ \mathsf{P} = \mathsf{id}_M$ and the hybrid Rota-Baxter rule. 
\end{enumerate}
\end{proposition}
\begin{proof} For the $(i) \Rightarrow (ii)$ direction, by definition  $\mathsf{D}$ and $\mathsf{P}$ satisfy FTC1, so $\mathsf{D} \circ \mathsf{P} = \mathsf{id}_M$ holds, and also satisfy FTC2, so by Prop \ref{prop:FTC2} they also satisfy the hybrid Rota-Baxter rule. 

For the $(ii) \Rightarrow (i)$ direction, we first use the hybrid Rota-Baxter rule to prove that the map $\mathsf{E}: A \to A$, defined as $\mathsf{E}(a) \colon \!\!\!= a - \mathsf{P}\left( \mathsf{D}(a) \right)$, preserves the multiplication. So we compute: 
\begin{gather*}
\mathsf{E}(ab) = ab - \mathsf{P}\left( \mathsf{D}(ab) \right) = ab - b~\mathsf{P}\left( \mathsf{D}(a) \right) - a ~\mathsf{P}\left( \mathsf{D}(b) \right) + \mathsf{P}\left( \mathsf{D}(a) \right)\mathsf{P}\left( \mathsf{D}(b) \right)\\
 =  \left( a -\mathsf{P}\left( \mathsf{D}(a) \right) \right)\left( b -\mathsf{P}\left( \mathsf{D}(b) \right) \right) = \mathsf{E}(a)\mathsf{E}(b)
\end{gather*}
So $\mathsf{E}(ab) = \mathsf{E}(a)\mathsf{E}(b)$. We will use this fact to show that $\mathsf{P}$ is an integration. First observe that from $\mathsf{D} \circ \mathsf{P} = \mathsf{id}_M$ and the Leibniz rule, we obtain that:
\begin{align*}
\mathsf{D}\left( \mathsf{P}(m) \mathsf{P}(n) \right) = \mathsf{P}(m) \mathsf{D}\left( \mathsf{P}(n) \right)   + \mathsf{P}(n) \mathsf{D}\left( \mathsf{P}(m) \right) = \mathsf{P}(m) n + \mathsf{P}(n) m 
\end{align*}
So $\mathsf{D}\left( \mathsf{P}(m) \mathsf{P}(n) \right) = \mathsf{P}(m) n + \mathsf{P}(n) m$. Also note that from $\mathsf{D} \circ \mathsf{P} = \mathsf{id}_M$, it easily follows that $\mathsf{E} \circ \mathsf{P} =0$. Then applying $\mathsf{P}$ to the previous equality, using $\mathsf{P}\left( \mathsf{D}(a) \right) = a - \mathsf{E}(a)$ and also that $\mathsf{E}$ preserves the multiplication, we compute that: 
\begin{gather*}
\mathsf{P}\left( \mathsf{P}(m) n \right) + \mathsf{P}\left( \mathsf{P}(n) m \right) =  \mathsf{P}\left( \mathsf{D}\left( \mathsf{P}(m) \mathsf{P}(n) \right) \right) = \mathsf{P}(m) \mathsf{P}(n) - \mathsf{E}\left( \mathsf{P}(m) \mathsf{P}(n) \right) = \\
\mathsf{P}(m) \mathsf{P}(n) -  \underbrace{\mathsf{E}\left( \mathsf{P}(m) \right)}_{=0} \underbrace{\mathsf{E}\left( \mathsf{P}(n) \right)}_{=0} =  \mathsf{P}(m) \mathsf{P}(n) 
\end{gather*}
So $\mathsf{P}$ satisfies the Rota-Baxter rule and is therefore an integration. By assumption, $\mathsf{D}$ and $\mathsf{P}$ satisfy FTC1, so it remains to explain why they also satisfy FTC2. However, by assumption $\mathsf{D}$ and $\mathsf{P}$ also satisfy the hybrid Rota-Baxter rule, and from FTC1 it follows that $\mathsf{D} \circ \mathsf{P} \circ \mathsf{D} = \mathsf{D}$. Thus by Prop \ref{prop:FTC2}, we have that $\mathsf{D}$ and $\mathsf{P}$ satisfy FTC2. So we conclude that $\xymatrix{ A \ar@/^/[r]^-{\mathsf{D}}  &  \ar@/^/[l]^-{\mathsf{P}}  M
  }$ is an FTC-pair. 
\end{proof}

Now an (commutative) integro-differential algebra (of weight $0$) \cite[Def 2.1]{guo2014integro} is a differential algebra $A$ with a derivation $\mathsf{D}: A \to A$ equipped with a linear map $\mathsf{P}: A \to A$ such that FTC1 and the hybrid Rota-Baxter rule hold. Thus by Prop \ref{prop:FTC-integro}, every integro-differential algebra is an FTC-pair, or more precisely: an integro-differential algebra is precisely the same thing as FTC-pair in the special case of an algebra seen as a module over itself. Therefore, this justifies our claim that an FTC-pair is indeed the ``module generalization'' of an integro-differential algebra. 

Before we give examples of FTC-pairs, let us quickly consider another case of interest, which is when the constants are precisely the base ring. For a commutative $k$-algebra $A$, let $\mathsf{u}: k \to A$ be the $k$-algebra structure morphism. Recall that for a derivation $\mathsf{D}: A \to M$, we have that $\mathsf{D} \circ \mathsf{u} = 0$. As such, $\mathsf{u}$ factors through the kernel of the derivation, so let $\hat{\mathsf{u}}: k \to \mathsf{ker}(\mathsf{D})$ be the canonical $k$-algebra morphism. 

\begin{definition}\label{def:aFTC} An FTC-pair $\xymatrix{ A \ar@/^/[r]^-{\mathsf{D}}  &  \ar@/^/[l]^-{\mathsf{P}}  M
  }$ is \textbf{augmented} if $\hat{\mathsf{u}}: k \to \mathsf{ker}(\mathsf{D})$ is an isomorphism, so $\mathsf{ker}(\mathsf{D}) \cong k$. 
\end{definition}

The name is justified by the fact that the algebra of an augmented FTC-pair is indeed an augmented algebra. 

\begin{lemma}  An FTC-pair $\xymatrix{ A \ar@/^/[r]^-{\mathsf{D}}  &  \ar@/^/[l]^-{\mathsf{P}}  M
  }$ is augmented if and only if $A$ has an augmentation $\mathsf{e}: A \to k$ such that the following equality holds for all $a \in A$: 
    \begin{equation}\label{def:aug-FTC2}\begin{gathered} 
\mathsf{u}\left( \mathsf{e}(a) \right) = a -\mathsf{P}\left( \mathsf{D}(a) \right) 
 \end{gathered}\end{equation}
\end{lemma}
\begin{proof} For the $\Rightarrow$ direction, by FTC2, we have a $k$-algebra morpshim $\hat{\mathsf{E}}: A \to \mathsf{ker}(\mathsf{D})$ defined as ${\hat{\mathsf{E}}(a) \colon \!\!\!= a - \mathsf{P}\left( \mathsf{D}(a) \right)}$. Then define $\mathsf{e}: A \to k$ as the composite $\mathsf{e} \colon \!\!\!= \hat{\mathsf{u}}^{-1} \circ \hat{\mathsf{E}}$. Since $\mathsf{e}$ is the composite of $k$-algebra morphisms, it is itself a $k$-algebra morphism. Therefore, $\mathsf{e}$ is an augmentation on $A$. Now note that for every $c \in \mathsf{ker}(\mathsf{D})$ that $\mathsf{u}\left( \hat{\mathsf{u}}^{-1}(c) \right) = c$. Therefore for every $a \in A$ we have that $\mathsf{u}\left( \mathsf{e}(a) \right) = \hat{\mathsf{E}}(a) = a - \mathsf{P}\left( \mathsf{D}(a) \right)$, and so (\ref{def:aug-FTC2}) holds as desired. 

For the $\Leftarrow$ direction, define $\hat{\mathsf{u}}^{-1}: \mathsf{ker}(\mathsf{D}) \to k$ as the restriction of the augmentation on the constants, that is, $\hat{\mathsf{u}}^{-1}(c) = \mathsf{e}(c)$. Since $\mathsf{e}$ is an augmentation, we have that $\mathsf{e} \circ \mathsf{u} = \mathsf{id}_k$, and so we get that $\hat{\mathsf{u}}^{-1} \circ \hat{\mathsf{u}} = \mathsf{id}_k$. On the other hand, (\ref{def:aug-FTC2}) tells us that $\mathsf{u}\left( \mathsf{e}(c) \right) = c$ for all $c \in \mathsf{ker}(\mathsf{D})$. Therefore, we have that $\hat{\mathsf{u}} \circ \hat{\mathsf{u}}^{-1} = \mathsf{id}_{\mathsf{ker}(\mathsf{D})}$. So, we conclude that our FTC-pair is indeed augmented. 
\hfill \end{proof}

It is important to highlight augmented FTC-pairs since, as we will see in Sec \ref{sec:equivalence}, they correspond to Zinbiel algebras over the base ring. 

\begin{example} \normalfont \label{ex:FTC-pairs} Here are some examples of FTC-pairs. 
\begin{enumerate}[{\em (i)}]
\item\label{ex:FTC-poly} Let $k$ be a field of characteristic zero, and let $k[x]$ be the polynomial ring over $k$. Then $\xymatrix{ k[x] \ar@/^/[r]^-{\mathsf{D}}  &  \ar@/^/[l]^-{\mathsf{P}}  k[x] 
  }$ is an augmented FTC-pair via the standard differentiation and integration of polynomials:
  \begin{align*}\mathsf{D}(x^n) = n x^{n-1} && \mathsf{P}(x^n) = \frac{1}{n+1}x^{n+1}
  \end{align*}
  This also makes $k[x]$ an integro-differential algebra. 
\item \label{ex:FTC-smooth} Let $\mathbb{R}$ be the reals and $\mathcal{C}^\infty(\mathbb{R})$ be the $\mathbb{R}$-algebra of smooth functions $\mathbb{R} \to \mathbb{R}$. Then $\xymatrix{ \mathcal{C}^\infty(\mathbb{R}) \ar@/^/[r]^-{\mathsf{D}}  &  \ar@/^/[l]^-{\mathsf{P}}  \mathcal{C}^\infty(\mathbb{R}) 
  }$ is an augmented FTC-pair via the standard differentiation and integration of smooth functions:
    \begin{align*}\mathsf{D}(f)(x) = f^\prime(x) && \mathsf{P}(f)(x) = \int \limits^x_0 f(t) \mathsf{d}t
  \end{align*}
    This also makes $\mathcal{C}^\infty(\mathbb{R})$ an integro-differential algebra \cite[Ex 2.2.(a)]{guo2014integro}. 
\item \label{ex:FTC-Hur} Here is an example of an FTC-pair which is not necessarily augmented and arises from an important example of an integro-differential algebra called the \textbf{Hurwitz series} (of weight $0$) algebra over an algebra \cite[Prop 3.2]{guo2014integro}, which turns out to also be the cofree differential algebra over an algebra. So let $k$ be a commutative ring and $A$ a commutative $k$-algebra. Let $\mathsf{H}(A)$ be the $k$-module of lists of elements of $A$, or in other words, the $k$-module of functions $f: \mathbb{N} \to A$, so $\mathsf{H}(A) = A^\mathbb{N}$. Then $\mathsf{H}(A)$ is a commutative $k$-algebra with multiplicaiton given by the \textbf{Hurwitz product} which for $f,g \in \mathsf{H}(A)$, $fg: \mathbb{N} \to A$ is defined as $(fg)(n) = \sum \limits^n_{k=0} f(k) g(n-k)$. Then $\mathsf{H}(A)$ is an integro-differential algebra, and so $\xymatrix{ \mathsf{H}(A) \ar@/^/[r]^-{\mathsf{D}}  &  \ar@/^/[l]^-{\mathsf{P}}  \mathsf{H}(A) 
  }$ is an FTC-pair where: 
  \begin{align*}\mathsf{D}(f)(n) = f(n+1) && \mathsf{P}(f)(n) = \begin{cases} 0 & \text{if } n=0 \\
  f(n-1) & \text{if } n\geq 1
  \end{cases}
  \end{align*}
  Visually, writing a list as $f = (a_0, a_1, a_2, \hdots)$, the derivation shifts the list to the left, $\mathsf{D}(a_0, a_1, a_2, \hdots) = (a_1, a_2, \hdots)$, while the integration shifts the list to the right by inserting $0$, $\mathsf{P}(a_0, a_1, a_2, \hdots) = (0,a_0, a_1, a_2, \hdots)$. With this, we clearly see that the $\mathsf{D}$-constants correspond precisely to lists of the form $(a, 0, 0, \hdots)$, so $\mathsf{ker}\left( \mathsf{D} \right) \cong A$. So this FTC-pair is augmented if and only if $A \cong k$. Furthermore, this example also clearly shows that we do not necessarily need to work with rationals to obtain interesting examples of integrations. That said, when $k$ is a field with characteristic zero, then $\mathsf{H}(A)$ is isomorphic as a $k$-algebra to the power series algebra over $A$, which provides an FTC-pair based on differentiating and integrating power series. So, while we need rationals to integrate polynomials or power series, they are not necessary to integrate Hurwitz series.  
\item \label{ex:FTC-shuffle} Here is an example of an FTC-pair which is not an integro-differential algebra, and is this time inspired by free Rota-Baxter algebras \cite[Chap 3]{guo2012introduction}. Just as with constructing free Rota-Baxter algebras, this example uses the \textbf{shuffle algebra} \cite[Sec 3.1.1]{guo2012introduction}. So, let $k$ be a commutative ring. For a $k$-module $V$, the shuffle algebra over $V$ is the commutative $k$-algebra $\mathsf{Sh}(V)$ defined as $\mathsf{Sh}(V) = \bigoplus \limits^\infty_{n=0} V^{\otimes^n} = k \oplus V \oplus (V \otimes V) \oplus \hdots$ with multiplication given by the \textbf{shuffle product} $\shuffle$ which is defined on pure tensors $\overline{v}= v_0 \otimes v_1 \otimes  \hdots \otimes v_n$ and $\overline{w}= w_0 \otimes w_1 \otimes \hdots \otimes w_m$ as follows: 
\begin{align*}
    1 \shuffle \overline{v} = \overline{v} = \overline{v} \shuffle 1 &&\overline{v} \shuffle \overline{w} = v_0 \otimes \left( (v_1 \otimes \hdots \otimes v_n) \shuffle \overline{w} \right) + w_0 \otimes \left( \overline{v} \shuffle (w_1 \otimes \hdots \otimes w_m) \right)
\end{align*}
Now consider the \emph{reduced} shuffle algebra $\mathsf{Sh}_+(V) = \bigoplus \limits^\infty_{n=1} V^{\otimes^n} = V \oplus (V \otimes V) \oplus \hdots$ \cite[Ex 5.2.(a)]{Loday2001}. While $\mathsf{Sh}_+(V)$ is a $\mathsf{Sh}(V)$-module as a subalgebra of $\mathsf{Sh}(V)$, there is another way in which it is a $\mathsf{Sh}(V)$-module. So define the module action $\triangleleft$ on pure tensors as follows: 
\begin{align*}
1 \triangleleft \overline{w} = \overline{w} &&  \overline{v} \triangleleft  \overline{w} = v_0 \otimes \left( (v_1 \otimes  \hdots \otimes v_n) \shuffle \overline{w} \right) 
\end{align*}
 Then $\xymatrix{ \mathsf{Sh}(V) \ar@/^/[r]^-{\mathsf{D}}  &  \ar@/^/[l]^-{\mathsf{P}}  \mathsf{Sh}_+(V) 
  }$ is an FTC-pair with the above $\mathsf{Sh}(V)$-module structure on $\mathsf{Sh}_+(V)$ and where the derivation projects out the degree $\geq 1$ terms, and the integration is the canonical injection of the reduced shuffle algebra into the shuffled algebra:
   \begin{align*}\mathsf{D}(1) =0 \qquad  \mathsf{D}\left( v_1 \otimes \hdots \otimes v_n \right) = v_1 \otimes \hdots \otimes v_n  &&  \mathsf{P}\left( v_1 \otimes \hdots \otimes v_n    \right) = v_1 \otimes \hdots \otimes v_n
  \end{align*}
That this is indeed an FTC-pair will be explained in Ex \ref{ex:FTC-shuffle2}. Moreover, in Ex \ref{ex:freeRB-FTC}, we will connect this FTC-pair to free Rota-Baxter algebras. 
\end{enumerate}
\end{example}

\begin{example} \normalfont \label{ex:FTC1/2-sept} Here are now some separating examples to show that FTC1 and FTC2 are indeed independent of each other. 
\begin{enumerate}[{\em (i)}]
\item Trivially, the zero maps $0: A \to M$ and $0: M \to A$ are respectively a derivation and an integration, and together satisfy FTC2. However, if $M$ is not the zero $A$-module, then they will not satisfy FTC1. 
\item An example of a differential Rota-Baxter algebra which is not an integro-differential algebra is given in \cite[Ex 2.2.(c)]{guo2014integro}. As such, this provides an example of a derivation and an integration which satisfy FTC1 but not FTC2. 
\end{enumerate}
\end{example}

We conclude this section by defining our category of (augmented) FTC-pairs, which in Sec \ref{sec:equivalence} we will show is equivalent to the category of Zinbiel algebras. 

\begin{definition}\label{def:FTC-category} For a commutative ring $k$, let $\mathsf{FTC}$ be the category where: 
\begin{enumerate}[{\em (i)}]
\item The objects of $\mathsf{FTC}$ are FTC-pairs $\xymatrix{ A \ar@/^/[r]^-{\mathsf{D}}  &  \ar@/^/[l]^-{\mathsf{P}}  M
  }$;
\item The maps of $\mathsf{FTC}$ are pairs $(f,g): \left( \xymatrix{ A \ar@/^/[r]^-{\mathsf{D}}  &  \ar@/^/[l]^-{\mathsf{P}}  M
  } \right) \to \left( \xymatrix{ C \ar@/^/[r]^-{\mathsf{B}}  &  \ar@/^/[l]^-{\mathsf{Q}}  N
  } \right)$ consisting of a $k$-algebra morphism ${f: A \to C}$ and an $A$-module morphism $g: M \to N$ in the sense that the following equality holds for all $a \in A$ and $m \in M$:
  \begin{equation}\label{gmodmap}\begin{gathered} 
g(am) = f(a) g(m)
 \end{gathered}\end{equation}
  and also that the following diagram commutes: 
 \begin{equation}\label{FTC-maps}\begin{gathered} 
\xymatrixcolsep{5pc} \xymatrix{ A \ar[d]_-{f} \ar@/^/[r]^-{\mathsf{D}}  &  \ar@/^/[l]^-{\mathsf{P}}  M \ar[d]^-{g} \\
 C \ar@/^/[r]^-{\mathsf{B}}  &  \ar@/^/[l]^-{\mathsf{Q}}  N
  }
 \end{gathered}\end{equation}
 Explicitly, the following equalities hold  for all $a \in A$ and $m \in M$:  
 \begin{align}\label{FTC-maps-equalities}
 g\left( \mathsf{D}(a) \right) = \mathsf{B}(f(a)) &&  f\left( \mathsf{P}(m) \right) =  \mathsf{Q}(g(m)) 
 \end{align}
\end{enumerate}
Composition in $\mathsf{FTC}$ is defined point-wise, that is, $(f,g) \circ (h,k) = (f \circ h, g \circ k)$, and identity maps in $\mathsf{FTC}$ are given by pairs of identity maps $(\mathsf{id}_A, \mathsf{id}_M)$. Also, let $\mathsf{FTC}_{aug}$ be the full subcategory of augmented FTC-pairs. 
\end{definition}

For our constructions in Sec \ref{sec:equivalence}, it will be useful to note what are isomorphisms of FTC-pairs, by which we mean isomorphisms in $\mathsf{FTC}$. It is straightforward to observe that an isomorphism of FTC-pairs then consists of an algebra isomorphism and a module isomorphism. We leave this as an exercise for the reader to check for themselves. 

\begin{lemma}\label{lem:FTC-iso} A map $(f,g) \in \mathsf{FTC}$ is an isomorphism if and only if $f$ and $g$ are isomorphisms. So in particular, if $(f,g) \in \mathsf{FTC}$, and both $f$ and $g$ are isomorphisms, then $(f^{-1}, g^{-1}) \in  \mathsf{FTC}$. 
\end{lemma}

\section{Zinbiel Algebras}

In this section, we briefly review Zinbiel algebras. For a more in-depth introduction to Zinbiel algebras, we refer the reader to \cite{alvarez2022algebraic,Loday2001,loday1995cup}. 

\begin{definition}\label{def:Zinbiel-algebra} Let $k$ be a commutative ring and $A$ a commutative $k$-algebra. An \textbf{$A$-Zinbiel algebra} \cite[Sec 7.1]{Loday2001} is an $A$-module $Z$ equipped with an $A$-bilinear binary operation $\triangleleft$, called the \textbf{Zinbiel operator}, which satisfies the \textbf{Zinbiel identity}, that is, the following equality holds for all $x,y,z \in Z$: 
\begin{equation}\label{def:Zinbiel}\begin{gathered} 
(x \triangleleft y) \triangleleft z = x \triangleleft (y \triangleleft z) + x \triangleleft (z \triangleleft y) 
 \end{gathered}\end{equation}
\end{definition}

Every Zinbiel algebra can be made into a commutative non-unital associative algebra. 

\begin{lemma}\label{lem:symprod} \cite[Prop 1.5]{loday1995cup} Let $Z$ be an $A$-Zinbiel algebra with Zinbiel operator $\triangleleft$. Then $Z$ is a commautive (non-unital) associative $A$-algebra with binary operation $\ast_\triangleleft$, called the \textbf{symmetrized product}, as follows for all $x,y \in Z$: 
\begin{equation}\label{def:symprod}\begin{gathered} 
x \ast_\triangleleft y = x \triangleleft y + y \triangleleft x 
\end{gathered}\end{equation}
\end{lemma}

\begin{example}\normalfont \label{ex:Zin-shuffle} The canonical example of a Zinbiel algebra is the \emph{reduced} shuffle algebra. Moreover, it is, in fact, the free Zinbiel algebra over a module \cite[Prop  1.7]{loday1995cup}. In fact, we have already provided the Zinbiel operator for the reduced shuffle algebra, it is the module action described in Ex \ref{ex:FTC-pairs}.(\ref{ex:FTC-shuffle}). Concretely, let $k$ be a commutative ring and $V$ be a $k$-module, and consider its reduced shuffle algebra $\mathsf{Sh}_+(V)$. Then $\mathsf{Sh}_+(V)$ is a $k$-Zinbiel algebra where the Zinbiel operator\footnote{We are abusing notation slightly here and use $\triangleleft$ for both the Zinbiel operator on $\mathsf{Sh}_+(V)$ and the action of $\mathsf{Sh}(V)$ on $\mathsf{Sh}_+(V)$. However, $\mathsf{Sh}(V)$ is not a Zinbiel algebra with $\triangleleft$ since $1 \triangleleft 1$ is undefined.} $\triangleleft$ is defined as follows on pure tensors $\overline{v}= v_0 \otimes v_1 \otimes  \hdots \otimes v_n$ and $\overline{w}$ as follows: 
 \[ \overline{v} \triangleleft \overline{w} = v_0 \otimes \left( (v_1 \otimes  \hdots \otimes v_n) \shuffle \overline{w} \right)   \]
 Equivalently, the Zinbiel operator can be directly defined in terms of $(m,n)$-shuffles \cite[Sec 1.7]{loday1995cup} (which are permutations that preserve the order of the first $m$ terms and the last $n$ terms). Moreover, the induced symmetrized product is precisely the shuffle product, $\ast_\triangleleft = \shuffle$. 
\end{example}
    
Another important way of constructing Zinbiel algebras, which is of particular interest to the story of this paper, is from integrations. 

\begin{lemma}\label{lem:P-Zinbiel} \cite[Prop 2.10]{uchino2008quantum} Let $\mathsf{P}: A \to M$ be an integration. Then $M$ is a $k$-Zinbiel algebra with binary operation $\triangleleft_\mathsf{P}$ defined as follows for all $m,n \in M$: 
\begin{equation}\label{def:P-Zinbiel}\begin{gathered} 
m \triangleleft_\mathsf{P} n \colon \!\!\!= \mathsf{P}(n) m 
 \end{gathered}\end{equation}
 The induced symmetrized product $\ast_{\triangleleft_\mathsf{P}}$ is worked out to be as follows for all $m,n \in M$: 
 \begin{equation}\label{def:P-Zinbiel-ast}\begin{gathered} 
m \ast_{\triangleleft_\mathsf{P}} n = \mathsf{P}(n) m + \mathsf{P}(m) n 
 \end{gathered}\end{equation}
\end{lemma}

\begin{example} \normalfont Let's apply Lemma \ref{lem:P-Zinbiel} to the integrations from Ex \ref{ex:FTC-pairs}.(\ref{ex:FTC-poly}) and (\ref{ex:FTC-smooth}) to obtain Zinbiel algebras. 
\begin{enumerate}[{\em (i)}]
\item Let $k$ be a field of characteristic zero. Then $k[x]$ is a $k$-Zinbiel algebra where:
\begin{align*}
    x^m \triangleleft_\mathsf{P} x^n = \frac{1}{n+1} x^{m+n+1} && x^m \ast_{\triangleleft_\mathsf{P}} x^n = \frac{m+n+2}{(m+1)(n+1)} x^{m+n+1}
\end{align*}
\item $\mathcal{C}^\infty(\mathbb{R})$ is an $\mathbb{R}$-Zinbiel algebra where:
\begin{align*}
(f \triangleleft_\mathsf{P} g)(x) =  f(x) \int \limits^x_0 g(t) \mathsf{d}t && (f \ast_{\triangleleft_\mathsf{P}} g)(x) = f(x) \int \limits^x_0 g(t) \mathsf{d}t + g(x) \int \limits^x_0 f(t) \mathsf{d}t
\end{align*}
\end{enumerate}
\end{example}

We conclude this section by describing the category of Zinbiel algebras, which we will show is equivalent to the category of FTC-pairs. 

\begin{definition}\label{def:Zinbiel-bigcat} Let $k$ be a commutative ring. Let $\mathsf{ZIN}$ be the category where: 
\begin{enumerate}[{\em (i)}]
\item The objects of $\mathsf{ZIN}$ are pairs $(A,Z)$, consisting of a commutative $k$-algebra $A$ and an $A$-Zinbiel algebra $Z$;
\item The maps of $\mathsf{ZIN}$ are pairs $(f,g): (A,Z) \to (A^\prime, Z^\prime)$ consisting of a $k$-algebra morphism ${f: A \to A^\prime}$ and an $A$-Zinbiel algebra morphism $g: Z \to Z^\prime$ in the sense that (\ref{gmodmap}) is satisfied and that the following equality holds: 
\begin{align}\label{def:Zin-map}
g( x \triangleleft y) = g(x) \triangleleft g(y)
\end{align}
 for all $x,y \in Z$.
\end{enumerate}
Composition in $\mathsf{ZIN}$ is defined point-wise, that is, $(f,g) \circ (h,k) = (f \circ h, g \circ k)$, and identity maps in $\mathsf{ZIN}$ are given by pairs of identity maps $(\mathsf{id}_A, \mathsf{id}_Z)$. 
\end{definition}

Again, for our constructions in Sec \ref{sec:equivalence}, it will be useful to note what are isomorphisms in $\mathsf{ZIN}$. We leave this as an exercise for the reader to check for themselves. 

\begin{lemma}\label{lem:ZIN-iso} A map $(f,g) \in \mathsf{ZIN}$ is an isomorphism if and only if $f$ and $g$ are isomorphisms. So in particular, if $(f,g) \in \mathsf{ZIN}$, and both $f$ and $g$ are isomorphisms, then $(f^{-1}, g^{-1}) \in \mathsf{ZIN}$. 
\end{lemma}

For a fixed algebra $A$, we can consider the subcategory of $\mathsf{ZIN}$ whose objects are pairs of the form $(A,Z)$ and whose maps are of the form $(1_A, g)$. An equivalent way of describing this subcategory is as the category of $A$-Zinbiel algebra defined in the expected way. In particular, we have $\mathsf{ZIN}_k$, the category of $k$-Zinbiel algebras, which we will show is equivalent to the category of augmented FTC-pairs. Explicitly: 

\begin{definition}\label{def:Zinbiel-category} Let $k$ be a commutative ring. Let $\mathsf{ZIN}_k$ be the category where: 
\begin{enumerate}[{\em (i)}]
\item The objects of $\mathsf{ZIN}_k$ are $k$-Zinbiel algebra $Z$;
\item The maps of $\mathsf{ZIN}_k$ are $k$-Zinbiel algebra morphisms, that is, $k$-module morphisms which also satisfy (\ref{def:Zin-map}). 
\end{enumerate}
Composition and identity maps are defined as for $k$-module morphisms. 
\end{definition}


\section{Equivalence}\label{sec:equivalence}

In this section, we prove the main result of this paper that the categories $\mathsf{FTC}$ and $\mathsf{ZIN}$ are equivalent. To do so, we must give functors $\mathsf{FTC} \to \mathsf{ZIN}$ and $\mathsf{ZIN} \to \mathsf{FTC}$, and natural isomorphisms $1_{\mathsf{FTC}}%
\Rightarrow \mathcal{G} \circ \mathcal{F}$ and $\mathcal{F} \circ \mathcal{G} \Rightarrow 1_{\mathsf{ZIN}}$. Throughout this section, fix a base commutative ring $k$. 

To define the functor $\mathsf{FTC} \to \mathsf{ZIN}$, we must first explain how from an FTC-pair we obtain a Zinbiel algebra. We have already seen in Lemma \ref{lem:P-Zinbiel} how from an integration we obtain a Zinbiel algebra over the base commutative ring. For an FTC-pair, we can improve upon this and instead obtain a Zinbiel algebra over the constants of the derivation. 

\begin{lemma}\label{lem:FTC-ZIN} If $\xymatrix{ A \ar@/^/[r]^-{\mathsf{D}}  &  \ar@/^/[l]^-{\mathsf{P}}  M
  }$ is an FTC-pair, then $M$ is a $\mathsf{ker}(\mathsf{D})$-Zinbiel algebra with Zinbiel operator $\triangleleft_\mathsf{P}$ defined as in Lemma \ref{lem:P-Zinbiel}.  
\end{lemma}
\begin{proof} It is a well-known fact that for a derivation $\mathsf{D}$, $\mathsf{ker}(\mathsf{D})$ is a sub-algebra of $A$. So $M$ is indeed a $\mathsf{ker}(\mathsf{D})$-module. Moreover, by Lemma \ref{lem:P-Zinbiel}, we already know that $\triangleleft_\mathsf{P}$ satisfies the Zinbiel identity (\ref{def:Zinbiel}). As such, it remains to explain why $\triangleleft_\mathsf{P}$ is $\mathsf{ker}(\mathsf{D})$-bilinear. So let $m,n \in M$ and $c \in \mathsf{ker}(\mathsf{D})$. Clearly, we have that $c(m \triangleleft n) = cm \triangleleft n$. So we need only show that $c(m \triangleleft n) = m \triangleleft cn$. To do so, we will show that the integration $\mathsf{P}$ is $\mathsf{ker}(\mathsf{D})$-linear. Recall that in the proof of Prop \ref{prop:FTC-integro}, we showed that $\mathsf{E} \circ \mathsf{P} =0$. Then using FTC1 and FTC2, we compute that: 
\begin{align*}
\mathsf{P}\left( cn \right) = \mathsf{P}\left( c \mathsf{D}\left( \mathsf{P}(n) \right) \right) = \mathsf{P}\left( \mathsf{D}\left( c \mathsf{P}(n) \right) \right) = c \mathsf{P}(n) - \mathsf{E}\left( c \mathsf{P}(n) \right) = c \mathsf{P}(n) -\mathsf{E}\left( c \right) \underbrace{\mathsf{E}\left( \mathsf{P}(n) \right)}_{=0} =  c\mathsf{P}(n)
\end{align*}
Thus $\mathsf{P}$ is $\mathsf{ker}(\mathsf{D})$-linear, and it follows that $c(m \triangleleft n) = m \triangleleft cn$. So we conclude that $\triangleleft_\mathsf{P}$ is $\mathsf{ker}(\mathsf{D})$-bilinear, and therefore $M$ is a $\mathsf{ker}(\mathsf{D})$-Zinbiel algebra as desired. 
\end{proof}

Let's explain how from a map in $\mathsf{FTC}$ we obtain a map in $\mathsf{ZIN}$. So let $(f,g): \left( \xymatrix{ A \ar@/^/[r]^-{\mathsf{D}}  &  \ar@/^/[l]^-{\mathsf{P}}  M
  } \right) \to \left( \xymatrix{ C \ar@/^/[r]^-{\mathsf{B}}  &  \ar@/^/[l]^-{\mathsf{Q}}  N
  } \right)$ be a map in $\mathsf{FTC}$. Note by (\ref{FTC-maps}), $g \circ \mathsf{D} = \mathsf{B} \circ f$ implies that if $c \in \mathsf{ker}(\mathsf{D})$, then $f(c) \in \mathsf{ker}(\mathsf{B})$. As such, restricting $f$ to $\mathsf{ker}(\mathsf{D})$, we obtain a $k$-algebra morphism $\overline{f}: \mathsf{ker}(\mathsf{D}) \to \mathsf{ker}(\mathsf{B})$ (explicitly, $\overline{f}(c) = f(c)$). 

  \begin{lemma}\label{lem:FTC-ZIN-map} If $(f,g)\!:\! \left( \xymatrix{ A \ar@/^/[r]^-{\mathsf{D}}  &  \ar@/^/[l]^-{\mathsf{P}}  M
  } \right) \to \left( \xymatrix{ C \ar@/^/[r]^-{\mathsf{B}}  &  \ar@/^/[l]^-{\mathsf{Q}}  N
  } \right)$ is a map in $\mathsf{FTC}$, then ${(\overline{f}, g)\!:\! (\mathsf{ker}(\mathsf{D}), M) \to (\mathsf{ker}(\mathsf{B}), N)}$ is a map in $\mathsf{ZIN}$. 
  \end{lemma}
  \begin{proof} Since $(f,g)$ satisfies (\ref{gmodmap}), it follows that $(\overline{f}, g)$ does as well. So it remains to show that $g$ preserves the Zinbiel operator, that is, $g\left( m \triangleleft_\mathsf{P} n \right) = g(m) \triangleleft_\mathsf{Q} g(n)$. Then using (\ref{gmodmap}) and (\ref{FTC-maps}), we compute that: 
  \begin{align*} 
  g\left( m \triangleleft_\mathsf{P} n \right) = g\left( \mathsf{P}(n) m \right) = f\left( \mathsf{P}(n) \right) g(m) = \mathsf{Q}(g(n) ) g(m) = g(m) \triangleleft_\mathsf{Q} g(n)
  \end{align*}
  So $g$ satisfies (\ref{def:Zin-map}), and therefore we conclude that $(\overline{f}, g)$ is a map in $\mathsf{ZIN}$.
  \end{proof}
  
Define the functor $\mathcal{F}\!:\! \mathsf{FTC} \to \mathsf{ZIN}$ on objects as $\mathcal{F}\left( \xymatrix{ A \ar@/^/[r]^-{\mathsf{D}}  &  \ar@/^/[l]^-{\mathsf{P}}  M
  }\right) \!\!\colon \!\!\!= \!(\mathsf{ker}(\mathsf{D}), M)$, and on maps as ${\mathcal{F}(f,g) \colon \!\!\!= (\overline{f}, g)}$. 

  \begin{proposition} $\mathcal{F}: \mathsf{FTC} \to \mathsf{ZIN}$ is a functor. 
  \end{proposition}
  \begin{proof} $\mathcal{F}$ is well-defined on objects by Lemma \ref{lem:FTC-ZIN} and also well-defined on maps by Lemma \ref{lem:FTC-ZIN-map}. It is straightforward to see that $\mathcal{F}$ also preserves composition and identities. Therefore, $\mathcal{F}$ is indeed a functor. 
  \end{proof}

To define the functor $\mathsf{ZIN} \to \mathsf{FTC}$, we first explain how to go from a Zinbiel algebra to an FTC-pair. So given an $A$-Zinbiel algebra $Z$, first define the commutative $k$-algebra $A \rtimes Z$ whose underlying $k$-module is $A \times Z$, and whose multiplication is defined as follows: 
\begin{align}
(a,x) (b,y) \colon \!\!\!= (ab, ay + bx + x \ast_\triangleleft y) 
\end{align}
where $\ast_\triangleleft y$ is defined as in Lemma \ref{lem:symprod}, and with unit $(1,0)$. Now $Z$ is an $A \rtimes Z$-module with module action defined as follows: 
\begin{align}
(a,x) y \colon \!\!\!= ay + y \triangleleft x
\end{align}

\begin{remark} \normalfont It is important to note that even though $Z$ is a non-unital sub-algebra of $A \rtimes Z$, the module action is different than simply multiplication, that is, $(a,x)y$ is not equal to $(a,x)(0,y)$.
\end{remark}

Now that we have our algebra $A \rtimes Z$ and our module $Z$, it remains to give the derivation and integration. Consider the projection onto the $Z$ part $\pi_Z: A \rtimes Z \to Z$, $\pi_Z(a,x) = x$, and similarly the injection $\iota_Z: Z \to A \rtimes Z$, $\iota_Z(x) = (0,x)$. This projection and injection will be the derivation and integration respectively. 

\begin{lemma}\label{lem:ZIN-FTC} For an $A$-Zinbiel algebra $Z$, $\xymatrix{ A \rtimes Z \ar@/^1pc/[r]^-{\pi_Z}  &  \ar@/^1pc/[l]^-{\iota_Z}  Z}$ is an FTC-pair. 
\end{lemma}
\begin{proof} We first show that $\pi_Z$ satisfies the Leibniz rule. On the one hand, we have that:  
\begin{align*}
\pi_Z\left( (a,x) (b,y) \right) = \pi_Z(ab, ay + bx + x \ast_\triangleleft y) = ay + bx + x \ast_\triangleleft y
\end{align*}
On the other hand, we have that: 
\begin{align*}
(a,x) \pi_Z\left(b,y \right) + (b,y) \pi_Z\left(a,x \right) = (a,x)y + (b,y) x = ay + y \triangleleft x + bx + x \triangleleft y = ay + bx + x \ast_\triangleleft y 
\end{align*}
Therefore, $\pi_Z\left( (a,x) (b,y) \right) = (a,x) \pi_Z\left(b,y \right) + (b,y) \pi_Z\left(a,x \right)$. As such, $\pi_Z$ is a derivation. Next, we show that $\iota_Z$ satisfies the Rota-Baxter rule. On the one hand, we have that: 
\begin{align*}
\iota_Z(x) \iota_Z(y) = (0,x) (0,y) = (0, x \ast_\triangleleft y)
\end{align*}
On the other hand, we have that: 
\begin{align*}
\iota_Z\left( \iota_Z(x) y \right) + \iota_Z\left( \iota_Z(y) z \right) = \iota_Z\left( (0,x) y \right) + \iota_Z\left( (0,y) x \right) = \iota_Z\left( y \triangleleft x \right) + \iota_Z\left( x \triangleleft y \right) = (0, y \triangleleft x) + (0, x \triangleleft y) = (0, x \ast_\triangleleft y)
\end{align*}
Therefore, $\iota_Z(x) \iota_Z(y) = \iota_Z\left( \iota_Z(x) y \right) + \iota_Z\left( \iota_Z(y) z \right)$. As such, $\iota_Z$ is an integration. 

Now we show that $\pi_Z$ and $\iota_Z$ satisfy FTC1 and FTC2. The former is straightforward since $\pi_Z(\iota_Z(x)) = \pi_Z(0,x)= x$, so $\pi_Z$ and $\iota_Z$ satisfy FTC1. For FTC2, first observe that $\iota_Z(\pi_Z(a,x)) = \iota_Z(x) = (0,x)$. Then define $c_{(a,x)} = (a,0)$, and note that $c_{(a,x)} \in \mathsf{ker}(\pi_Z)$. Therefore, we have that $\iota_Z(\pi_Z(a,x)) = (a,x) - c_{(a,x)}$, so (\ref{FTC2}) is satisfied. Lastly, it is straightforward to see that $c_{(a,x)(b,y)} = c_{(a,x)} c_{(b,y)}$. As such, $\pi_Z$ and $\iota_Z$ satisfy FTC2. Thus we conclude that $\xymatrix{ A \rtimes Z \ar@/^1pc/[r]^-{\pi_Z}  &  \ar@/^1pc/[l]^-{\iota_Z}  Z}$ is an FTC-pair as desired. 
\end{proof}

\begin{example}\normalfont \label{ex:FTC-shuffle2} Let's apply this construction to the reduced shuffle algebra from Ex \ref{ex:Zin-shuffle}. So let $V$ be a $k$-module. Observe that since the induced symmetrized product is precisely the shuffle product, we have an isomorphism of $k$-algebra $k \rtimes \mathsf{Sh}_+(V) \cong \mathsf{Sh}(V)$. Moreover, the resulting $\mathsf{Sh}(V)$-module structure on $\mathsf{Sh}_+(V)$ is precisely the one given in Ex \ref{ex:FTC-pairs}.(\ref{ex:FTC-shuffle}). Therefore, applying Lemma \ref{lem:ZIN-FTC} to $\mathsf{Sh}_+(V)$ results (up to isomorphism) to the FTC-pair $\xymatrix{ \mathsf{Sh}(V) \ar@/^/[r]^-{\mathsf{D}}  &  \ar@/^/[l]^-{\mathsf{P}}  \mathsf{Sh}_+(V) 
  }$ from Ex \ref{ex:FTC-pairs}.(\ref{ex:FTC-shuffle}). 
\end{example}

Turning back our attention to constructing our desired functor, we must explain how it acts on maps. Now let $(f,g): (A,Z) \to (A^\prime, Z^\prime)$ be a map in $\mathsf{ZIN}$. Then define the map $f \rtimes g: A \rtimes Z \to A^\prime \rtimes Z^\prime$ as ${(f\rtimes g)(a,x) = (f(a), g(x))}$. 

\begin{lemma}\label{lem:ZIN-FTC-map} If $(f,g): (A,Z) \to (A^\prime, Z^\prime)$ is a map $\mathsf{ZIN}$, then $(f \rtimes g, g): \left( \xymatrix{ A \rtimes Z \ar@/^1pc/[r]^-{\pi_Z}  &  \ar@/^1pc/[l]^-{\iota_Z}  Z} \right) \to \left( \xymatrix{ A^\prime \rtimes Z^\prime \ar@/^1pc/[r]^-{\pi_{Z^\prime}}  &  \ar@/^1pc/[l]^-{\iota_{Z^\prime}}  Z^\prime} \right)$ is a map in $\mathsf{FTC}$. 
\end{lemma}
\begin{proof} We first show that $f \rtimes g$ is a $k$-algebra morphism. Since $f$ is a $k$-algebra morphism, we get that $(f \rtimes g)(1,0) = (f(1), 0) = (1,0)$, thus $f \rtimes g$ preserves the unit. For the multiplication, observe that by (\ref{def:Zin-map}), since $g$ preserves $\triangleleft$, it also preserves $\ast_{\triangleleft}$, that is, $g(x \ast_\triangleleft y) = g(x) \ast_\triangleleft g(y)$. Then using this, the fact that $f$ preserves the multiplication, and also (\ref{gmodmap}), we then compute that: 
\begin{gather*} (f \rtimes g)(a,x)(f \rtimes g)(b,y) = (f(a), g(x)) (f(b), g(y)) = (f(a)f(b), f(a)g(y) + f(b) g(x) + g(x) \ast_\triangleleft g(y) ) \\
=(f(ab), g(ay)  + g(bx) + g(x \ast_\triangleleft y) ) = (f(ab),  g(ay + b + x \ast_\triangleleft y) ) = (f \rtimes g)(ab, ay + b + x \ast_\triangleleft y) =  (f \rtimes g)\left( (a,x)(b,y) \right)
\end{gather*}
So $f \rtimes g$ also preserves the multiplication, and so we have that $f \rtimes g$ is indeed a $k$-algebra morphism. Next, we show that $f \rtimes g$ and $g$ satisfy (\ref{gmodmap}) together. Then using (\ref{gmodmap}) for $(f,g)$ and also (\ref{def:Zin-map}), we compute that: 
\[ g\left( (a,x) y \right) = g( ay + y \triangleleft x) = f(a) g(y) + g(y) \triangleleft g(x) = (f(a), g(x)) g(y) =  (f \rtimes g)(a,x) g(y)  \]
So $g\left( (a,x) y \right) =  (f \rtimes g)(a,x) g(y)$. Lastly, it remains to show that $(f\rtimes, g)$ is also compatible with the derivation and integration, that is, that the equalities in (\ref{FTC-maps-equalities}) hold. However, we easily compute that: 
\[ g\left( \pi_Z(a,x) \right) = g(x) = \pi_{Z^\prime}(f(a),g(x) ) = \pi_{Z^\prime}\left(  (f \rtimes g)(a,x) \right) \]
\[ (f \rtimes g) \left( \iota_Z(x) \right) = (f \rtimes g) (0,x) = (0, g(x)) = \iota_{Z^\prime}(g(x))  \]
Thus $ g\left( \pi_Z(a,x) \right) = \pi_{Z^\prime}\left(  (f \rtimes g)(a,x) \right)$ and $(f \rtimes g) \left( \iota_Z(x) \right) = \iota_{Z^\prime}(g(x)) $. As such, we conclude that $(f \rtimes g, g)$ is a map in $\mathsf{FTC}$, as desired. 
\end{proof}

Finally, we define the functor $\mathcal{G}: \mathsf{ZIN} \to \mathsf{FTC}$ on objects as $\mathcal{G}(A,Z) \colon \!\!\!= \xymatrix{ A \rtimes Z \ar@/^1pc/[r]^-{\pi_Z}  &  \ar@/^1pc/[l]^-{\iota_Z}  Z
  }$, and on maps as $\mathcal{G}(f,g) \colon \!\!\!= (f \rtimes g, g)$. 

    \begin{proposition} $\mathcal{G}: \mathsf{ZIN} \to \mathsf{FTC}$ is a functor. 
  \end{proposition}
  \begin{proof} $\mathcal{G}$ is well-defined on objects by Lemma \ref{lem:ZIN-FTC} and also well-defined on maps by Lemma \ref{lem:ZIN-FTC-map}. It is straightforward to see that $\mathcal{G}$ also preserves composition and identities. Therefore, $\mathcal{G}$ is indeed a functor.
  \end{proof}

Now that we have our functors $\mathcal{F}: \mathsf{FTC} \to \mathsf{ZIN}$ and $\mathcal{G}: \mathsf{ZIN} \to \mathsf{FTC}$, we wish to show that they give an equivalence. So we must now construct natural isomorphisms $1_{\mathsf{FTC}}%
\Rightarrow \mathcal{G} \circ \mathcal{F}$ and $\mathcal{F} \circ \mathcal{G} \Rightarrow 1_{\mathsf{ZIN}}$.

So let us first build our natural isomorphism $1_{\mathsf{FTC}}%
\Rightarrow \mathcal{G} \circ \mathcal{F}$. To do so, starting with an FTC-pair $\xymatrix{ A \ar@/^/[r]^-{\mathsf{D}}  &  \ar@/^/[l]^-{\mathsf{P}}  M }$, we first observe that: 
\[\mathcal{G}\left( \mathcal{F}\left( \xymatrix{ A \ar@/^/[r]^-{\mathsf{D}}  &  \ar@/^/[l]^-{\mathsf{P}}  M
  }\right) \right) = \xymatrix{ \mathsf{ker}(\mathsf{D}) \rtimes M \ar@/^1pc/[r]^-{\pi_M}  &  \ar@/^1pc/[l]^-{\iota_M}  M}\] So define $\eta: \left( \xymatrix{ A \ar@/^/[r]^-{\mathsf{D}}  &  \ar@/^/[l]^-{\mathsf{P}}  M } \right) \to \mathcal{G}\left( \mathcal{F}\left( \xymatrix{ A \ar@/^/[r]^-{\mathsf{D}}  &  \ar@/^/[l]^-{\mathsf{P}}  M
  }\right) \right)$ as the pair $\eta \colon \!\!\!= (\eta_1, \eta_2)$ where $\eta_1: A \to \mathsf{ker}(\mathsf{D}) \rtimes M$ and ${\eta_2: M \to M}$ are defined respectively as follows: 
  \begin{align} \eta_1(a) = \left( a - \mathsf{P}\left( \mathsf{D}(a) \right), \mathsf{D}(a) \right) && \eta_2(m) = m
  \end{align}
where note that $\eta_1$ is well-defined by FTC2. Now, define its inverse $\eta^{-1}:  \mathcal{G}\left( \mathcal{F}\left( \xymatrix{ A \ar@/^/[r]^-{\mathsf{D}}  &  \ar@/^/[l]^-{\mathsf{P}}  M
  }\right) \right) \to \left( \xymatrix{ A \ar@/^/[r]^-{\mathsf{D}}  &  \ar@/^/[l]^-{\mathsf{P}}  M } \right)$ as the pair $\eta^{-1} \colon \!\!\!= (\eta^{-1}_1, \eta^{-1}_2)$ where $\eta^{-1}_1:  \mathsf{ker}(\mathsf{D}) \rtimes M \to A$ and $\eta^{-1}_2: M \to M$ are defined as follows: 
  \begin{align} \eta^{-1}_1(c,m) = c + \mathsf{P}(m) && \eta^{-1}_2(m) = m
  \end{align}
  
\begin{lemma} $\eta: 1_{\mathsf{FTC}}%
\Rightarrow \mathcal{G} \circ \mathcal{F}$ is a natural isomorphism. 
\end{lemma}
\begin{proof} We first need to show that $\eta$ is well-defined, that is, that $(\eta_1, \eta_2)$ is a map in $\mathsf{FTC}$. So first, let us show that $\eta_1$ is indeed a $k$-algebra morphism. Now writing $\mathsf{E}(a) = a - \mathsf{P}\left( \mathsf{D}(a) \right)$, so $\eta_1(a) = \left( \mathsf{E}(a), \mathsf{D}(a) \right)$. Recall that by Prop \ref{prop:FTC2}, $\mathsf{E}$ was a $k$-algebra morphism. As such, we first easily see that $\eta_1(1) = \left( \mathsf{E}(1), \mathsf{D}(1) \right) = (1,0)$. So $\eta_1$ preserves the unit. Now, using the Leibniz rule, we can also compute that: 
\begin{gather*} \eta_1(a) \eta_1(b) = \left( \mathsf{E}(a), \mathsf{D}(a) \right)\left( \mathsf{E}(b), \mathsf{D}(b) \right) = \left( \mathsf{E}(a)\mathsf{E}(b), \mathsf{E}(a)\mathsf{D}(b) + \mathsf{D}(a)\mathsf{E}(b) + \mathsf{D}(a) \ast_{\triangleleft_{\mathsf{P}}} \mathsf{D}(b) \right) \\
= \left( \mathsf{E}(ab), \mathsf{E}(a)\mathsf{D}(b) + \mathsf{D}(a)\mathsf{E}(b)  + \mathsf{P}\left( \mathsf{D}(b) \right)\mathsf{D}(a) +  \mathsf{P}\left( \mathsf{D}(a) \right)\mathsf{D}(b) \right) \\
= \left( \mathsf{E}(ab), \left(\mathsf{P}\left( \mathsf{D}(a) \right) + \mathsf{E}(a) \right) \mathsf{D}(b) + \left(\mathsf{P}\left( \mathsf{D}(b) \right) + \mathsf{E}(b) \right) \mathsf{D}(a) \right) \\
= \left( \mathsf{E}(ab), a\mathsf{D}(b) +b \mathsf{D}(a) \right) = \left( \mathsf{E}(ab), \mathsf{D}(ab) \right) = \eta_1(ab) 
\end{gather*}
So $\eta_1$ preserves the multiplication as well, and thus $\eta_1$ is indeed a $k$-algebra morphism. Next, we need to show that $\eta_2$ is an $A$-module morphism, that is, that (\ref{gmodmap}) holds. So we compute: 
\[ \eta_1(a)\eta_2(m) = \left( \mathsf{E}(a), \mathsf{D}(a) \right) m = \mathsf{E}(a)m + m \triangleleft_\mathsf{P} \mathsf{D}(a) = \mathsf{E}(a)m + \mathsf{P}\left( \mathsf{D}(a) \right) m =  \left(\mathsf{P}\left( \mathsf{D}(a) \right) + \mathsf{E}(a) \right) m = am = \eta_2(am) \]
We now show that $\eta_1$ and $\eta_2$ are compatible with the derivations and integrations, that is, that the equalities of (\ref{FTC-maps-equalities}) hold. For the derivations, this is straightforward: 
\[ \pi_M(\eta_1(a)) = \pi_M \left( \mathsf{E}(a), \mathsf{D}(a) \right) =  \mathsf{D}(a) =  \eta_2\left( \mathsf{D}(a) \right) \]
For the integration, recall in the proof of Lemma \ref{lem:FTC-ZIN}, we proved that $\mathsf{E} \circ \mathsf{P} =0$. Then, using this fact and FTC1, we compute that: 
\[ \eta_1(\mathsf{P}(m)) = \left( \mathsf{E}\left(\mathsf{P}(m) \right),  \mathsf{D}\left(\mathsf{P}(m) \right) \right) = (0,m) = \iota_M(m) = \iota_M(\eta_2(m)) \]
Thus we conclude that $\eta = (\eta_1, \eta_2)$ is indeed a map in $\mathsf{FTC}$. 

Next, we show that $\eta$ is natural. So let $(f,g): \left( \xymatrix{ A \ar@/^/[r]^-{\mathsf{D}}  &  \ar@/^/[l]^-{\mathsf{P}}  M
  } \right) \to \left( \xymatrix{ C \ar@/^/[r]^-{\mathsf{B}}  &  \ar@/^/[l]^-{\mathsf{Q}}  N
  } \right)$ be a map in $\mathsf{FTC}$. Then we have that $\mathcal{G}\left( \mathcal{F}\left( f,g \right) \right) = (\overline{f} \rtimes g, g)$. Now let $\mathsf{E}^\prime(c) = c - \mathsf{Q}\left( \tilde{\mathsf{L}}(c) \right)$. Then by (\ref{FTC-maps-equalities}), we have that $f \circ \mathsf{E} = \mathsf{E}^\prime \circ f$. Then using this and (\ref{FTC-maps-equalities}) again, we compute that: 
  \begin{gather*} 
  (\overline{f} \rtimes g)(\eta_1(a)) \!= \!  (\overline{f} \rtimes g) \left( \mathsf{E}(a), \mathsf{D}(a) \right) = \left( \overline{f}\left( \mathsf{E}(a) \right), g\left(\mathsf{D}(a) \right)\right) = \left( f\left( \mathsf{E}(a) \right), g\left(\mathsf{D}(a) \right)\right) = \left( \mathsf{E}^\prime\left( f(a) \right), \tilde{\mathsf{L}}(f(a)) \right) = \eta_1(f(a))
  \end{gather*}
So $ (\overline{f} \rtimes g) \circ \eta_1 = \eta_1 \circ f$, and clearly we have that $g \circ \eta_2 = \eta_2 \circ g$. Together, these give us that $\mathcal{G}\left( \mathcal{F}\left( f,g \right) \right) \circ \eta = \eta \circ (f,g)$, and so $\eta$ is indeed a natural transformation.

Lastly, we need to explain why $\eta$ is also an isomorphism. By Lemma \ref{lem:FTC-iso}, it suffices to show that $\eta_1$ and $\eta_2$ are isomorphisms. Trivially, $\eta_2$ is an isomorphism. For $\eta_1$, that $\eta^{-1}_1 \circ \eta_1 = \mathsf{id}_A$ we easily compute that:  
\[ \eta^{-1}_1(\eta_1(a)) = \eta^{-1}_1\left( \mathsf{E}(a), \mathsf{D}(a) \right)= \mathsf{E}(a) + \mathsf{P}\left( \mathsf{D}(a) \right) = a \]
For the other direction, that $\eta_1 \circ \eta^{-1}_1 = \mathsf{id}_{\mathsf{ker}\left( \mathsf{D} \right) \rtimes M}$, we compute that: 
\[ \eta_1(\eta^{-1}_1(c,m)) = \eta_1\left( c +  \mathsf{P}(m) \right) = (\mathsf{E}(c) + \underbrace{\mathsf{E}\left( \mathsf{P}(m) \right)}_{=0}, \underbrace{\mathsf{D}(c)}_{=0} + \mathsf{D}\left( \mathsf{P}(m) \right) = (c,m)  \]
So $\eta_1$ is an isomorphism. Therefore, we conclude that $\eta$ is an isomorphism, and therefore that $\eta^{-1}  = (\eta^{-1}_1, \eta^{-1}_2)$ is a map in $\mathsf{FTC}$. 
\end{proof}

Next, let us build our natural isomorphism of type $\mathcal{F} \circ \mathcal{G} \Rightarrow 1_{\mathsf{ZIN}}$. So starting with an $(A,Z) \in \mathsf{ZIN}$, we have that:
\[\mathcal{F}\left( \mathcal{G}\left( A,Z\right) \right) = (\mathsf{ker}(\pi_Z), Z)\] 
Unpacking this a bit more, define $\epsilon: \mathcal{F}\left( \mathcal{G}\left( A,Z\right) \right) \to (A,Z)$ as the pair $\epsilon \colon \!\!\!=( \epsilon_1, \epsilon_2)$ where $\epsilon_1: \mathsf{ker}(\pi_Z) \to A$ and ${\epsilon_2: Z \to Z}$ are defined as follows: 
\begin{align} \epsilon_1(a,0) = a && \epsilon_2(x) = x
  \end{align}
  Define its inverse $\epsilon^{-1}: (A,Z) \to \mathcal{F}\left( \mathcal{G}\left( A,Z\right) \right) $ as the pair $\epsilon^{-1} \colon \!\!\!=( \epsilon^{-1}_1, \epsilon^{-1}_2)$ where $\epsilon^{-1}_1: A \to \mathsf{ker}(\pi_Z)$ and ${\epsilon^{-1}_2: Z \to Z}$ are defined as follows: 
\begin{align} \epsilon^{-1}_1(a) = (a,0) && \epsilon^{-1}_2(x) = x
  \end{align}

  \begin{lemma} $\epsilon: \mathcal{F} \circ \mathcal{G} \Rightarrow 1_{\mathsf{ZIN}}$ is a natural isomorphism. 
\end{lemma}
\begin{proof} We first need to explain why $\epsilon$ is well-defined, that is, that $(\epsilon_1, \epsilon_2)$ is a map in $\mathsf{ZIN}$. It is straightforward to see that $\epsilon_1$ is indeed a $k$-algebra morphism, and also that $\epsilon_2$ is an $A$-module morphism, that is, that (\ref{gmodmap}) holds. Now observe that: 
\[ x \triangleleft_{\iota_Z} y = \iota_Z(y) x = (0,y) x = x \triangleleft y \]
So the induced Zinbiel algebra structure on $Z$ is the same as the original one, and so clearly $\epsilon_2$ satisfies (\ref{def:Zin-map}). Thus have that $\epsilon = (\epsilon_1, \epsilon_2)$ is indeed a map in $\mathsf{ZIN}$. 

Next, we show that $\epsilon$ is natural. So let $(f,g): (A,Z) \to (A^\prime, Z^\prime)$ be a map in $\mathsf{ZIN}$. Then we have that $\mathcal{F}\left( \mathcal{G}\left( f,g \right) \right) = (\overline{f \rtimes g}, g)$. Now, we first easily compute that: 
  \begin{gather*} 
\epsilon_1\left( \overline{f \rtimes g}(a,0) \right) = \epsilon_1\left( (f \rtimes g)(a,0)   \right) = \epsilon_1(f(a), 0) = f(a) = f(\epsilon_1(a,0)) 
  \end{gather*}
So $\epsilon_1 \circ \overline{f \rtimes g} = f \circ \epsilon_1$, and clearly we have that $\epsilon_2 \circ g = g \circ \epsilon_2$. Together, these give us that $\epsilon \circ \mathcal{F}\left( \mathcal{G}\left( f,g \right) \right)= (f,g) \circ \epsilon$, and so $\epsilon$ is indeed a natural transformation.

Lastly, we need to explain why $\epsilon$ is also an isomorphism. By Lemma \ref{lem:ZIN-iso}, it suffices to show that $\epsilon_1$ and $\epsilon_2$ are isomorphisms. However, they are clearly isomorphism with inverses $\epsilon^{-1}_1$ and $\epsilon^{-1}_2$ respectively. Therefore, we conclude that $\epsilon$ is an isomorphism, and also that $\epsilon^{-1}  = (\epsilon^{-1}_1, \epsilon^{-1}_2)$ is a map in $\mathsf{ZIN}$.  
\end{proof}

Bringing all of this together, we obtain our main result: 

\begin{theorem}\label{thm:FTC-Zin} There is an equivalence of categories $\mathsf{FTC} \simeq \mathsf{ZIN}$, given by the functors $\mathcal{F}$ and $\mathcal{G}$, and the natural isomorphisms $\eta$ and $\epsilon$. 
\end{theorem}

Let us consider the equivalence on augmented FTC-pairs and $k$-Zinbiel algebras. For an augmented FTC-pair $\xymatrix{ A \ar@/^/[r]^-{\mathsf{D}}  &  \ar@/^/[l]^-{\mathsf{P}}  M
  }$, since $k \cong \mathsf{ker}\left( \mathsf{D} \right)$, we have that $\mathcal{F}\left(\xymatrix{ A \ar@/^/[r]^-{\mathsf{D}}  &  \ar@/^/[l]^-{\mathsf{P}}  M
  } \right) \cong (k, M)$. On the other hand, for a $k$-Zinbiel algebra $Z$, since $\mathsf{ker}\left( \pi_Z \right) \cong k$, we have that $\mathcal{G}\left(k,Z\right) =  \xymatrix{ k \rtimes Z \ar@/^1pc/[r]^-{\pi_Z}  &  \ar@/^1pc/[l]^-{\iota_Z}  Z
  }$ which is an augmented FTC-pair. Thus, we obtain an equivalence between the category of augmented FTC-pairs and the category of $k$-Zinbiel algebras, as desired. 

\begin{corollary}\label{cor:aFTC-Zin} There is an equivalence of categories $\mathsf{FTC}_{aug} \simeq \mathsf{ZIN}_k$. 
\end{corollary}

As mentioned in the introduction, generalizing these main results to the non-commutative setting is straightforward, which gives us an equivalence between the category of dendriform algebras and the category of the non-commutative versions of FTC pairs. 

\section{Other Constructions of FTC-Pairs}\label{sec:constructions}

In this final section, we provide a way of constructing an FTC-pair from an integration, and also a way of constructing an FTC-pair from a derivation. So, throughout this section, fix a base commutative ring $k$. 

Constructing an FTC-pair from an integration is straightforward since we already know how to construct a Zinbiel algebra from an integration. Indeed, given an integration $\mathsf{P}: A \to M$, by Lemma \ref{lem:P-Zinbiel}, we know that $M$ is a $k$-Zinbiel algebra. In other words, $M$ is an object in $\mathsf{ZIN}_k$. Thus passing $M$ through the equivalence $\mathsf{ZIN}_k \simeq \mathsf{FTC}_{aug}$ results in the augmented FTC-pair:

\begin{lemma}\label{lem:int-FTC} If $\mathsf{P}: A \to M$ is an integration, then $\xymatrix{k \rtimes M \ar@/^1pc/[r]^-{\pi_M}  &  \ar@/^1pc/[l]^-{\iota_M}  M
  }$ is an augmented FTC-pair. 
\end{lemma}

So from any integration, we can build a derivation and a new integration which together satisfy the Fundamental Theorems of Calculus. In particular, we can apply this construction to any Rota-Baxter algebra to obtain an augmented FTC-pair. 

\begin{example} \label{ex:freeRB-FTC} \normalfont Let's apply Lemma \ref{lem:int-FTC} to a free Rota-Baxter algebra \cite[Thm 3.2.1]{guo2012introduction} (another presentation of the free Rota-Baxter algebra construction can be found in \cite[Ex 4]{Blute2019}). For a commutative $k$-algebra $A$, let $\mathsf{RB}(A)$ be the $k$-algebra whose underlying $k$-module is the reduced shuffle algebra $\mathsf{RB}(A) = \mathsf{Sh}_+(A)$, and whose multiplication is given by the \textbf{augmented mixable shuffle product} $\cshuffle$ which is defined on pure tensors $\overline{a}= a_0 \otimes a_1 \otimes  \hdots \otimes a_n$ and $\overline{b}= b_0 \otimes b_1 \otimes \hdots \otimes b_m$ as follows: 
\begin{align*}
    \overline{a} \cshuffle \overline{b} = a_0b_0 \otimes \left( (a_1 \otimes  \hdots \otimes a_n) \shuffle (b_1 \otimes  \hdots \otimes b_m) \right) 
\end{align*}
Then $\mathsf{RB}(A)$ is the free (commutative) Rota-Baxter algebra over $A$ with integration $\mathsf{P}: \mathsf{RB}(A) \to \mathsf{RB}(A)$ defined as follows:
 \begin{align*}
  \mathsf{P}(a_0 \otimes 1) = 1 \otimes a_0 &&   \mathsf{P}\left( a_0 \otimes a_1 \otimes  \hdots \otimes a_n \right) = 1 \otimes  a_0 \otimes a_1 \otimes  \hdots \otimes a_n
  \end{align*}
 However, in general, $\mathsf{RB}(A)$ does not have a derivation which makes it an integro-differential algebra. So, to obtain an FTC-pair, we must instead first consider the induced Zinbiel algebra. It turns out, as explained in the proof of \cite[Thm 5.2.4]{guo2012introduction}, that applying Lemma \ref{lem:P-Zinbiel} to $\mathsf{P}: \mathsf{RB}(A) \to \mathsf{RB}(A)$ gives precisely the Zinbiel algebra structure on the reduced shuffle algebra $\mathsf{Sh}_+(A)$ from Ex \ref{ex:Zin-shuffle}. Then by Ex \ref{ex:FTC-shuffle2}, it follows that applying Lemma \ref{lem:int-FTC} to ${\mathsf{P}: \mathsf{RB}(A) \to \mathsf{RB}(A)}$ gives us (up to isomorphism) the augmented FTC-pair $\xymatrix{ \mathsf{Sh}(A) \ar@/^/[r]^-{\mathsf{D}}  &  \ar@/^/[l]^-{\mathsf{P}}  \mathsf{Sh}_+(A) 
  }$ from Ex \ref{ex:FTC-pairs}.(\ref{ex:FTC-shuffle}). Note, however, that here we started with an algebra, while in Ex \ref{ex:FTC-pairs}.(\ref{ex:FTC-shuffle}) we needed only a module. 
\end{example}

On the other hand, constructing an FTC-pair from a derivation requires more setup. While this setting and construction may seem somewhat ad hoc, it in fact arises quite naturally for many well-known derivations, as we will see in Ex \ref{ex:der-FTC}. Moreover, this setup appears to generalize nicely in any differential category with antiderivatives\footnote{In fact, the inspiration for the notion $\mathsf{D}^\circ$, $\mathsf{L}$, and $\mathsf{K}$ come from the notation used in a differential category with antiderivatives.} \cite{cockett_lemay_2018}. As such, in a follow-up paper, we hope to generalize this construction to build FTC-pairs which involve the (co)free (co)algebras of the (co)monad a differential category with antiderivatives. 

So let's start with a derivation $\mathsf{D}: A \to M$. We also need to have an $A$-linear morphism of dual type ${\mathsf{D}^\circ: M \to A}$, in particular this means that $\mathsf{D}^\circ(am) = a \mathsf{D}^\circ(m)$ for all $a\in A$ and $m \in M$. The last ingredient we need is an idempotent $k$-algebra morphism $\mathsf{E}: A \to A$ such that $\mathsf{E} \circ \mathsf{D}^\circ =0$ and $\mathsf{D} \circ \mathsf{E} =0$. 

The algebra component of our FTC-pair will be $A$, while the module component will be $\mathsf{ker}(\mathsf{E})$ (which is indeed an $A$-module since it is also a subalgebra of $A$). For the derivation, first define the map $\mathsf{L}: A \to A$ as the composite $\mathsf{L} := \mathsf{D}^\circ \circ \mathsf{D}$. Recall that post-composing a derivation by a module morphism results again in a derivation. Therefore, $\mathsf{L}: A \to A$ is a derivation, which makes $A$ into a differential algebra. Now since $\mathsf{E} \circ \mathsf{D}^\circ =0$, it follows that $\mathsf{E} \circ \mathsf{L} =0$. Therefore, we obtain the map $\tilde{\mathsf{L}}: A \to \mathsf{ker}(\mathsf{E})$ defined as $\tilde{\mathsf{L}}(a) = \mathsf{L}(a)$, which will be the derivation of the FTC-pair. For the integration, we will need to make another assumption. First define the map $\mathsf{K}: A \to A$ as $\mathsf{K}(a) = \mathsf{L}(a) + \mathsf{E}(a)$. Now suppose that $\mathsf{K}$ is an isomorphism. Then define $\mathsf{P}: \mathsf{ker}(\mathsf{E}) \to A$ as $\mathsf{P}(a) = \mathsf{K}^{-1}(a)$, and this will be the integration of our FTC-pair. 

\begin{lemma}\label{lem:der-FTC} Let $\mathsf{D}: A \to M$ be a derivation, $\mathsf{D}^\circ: M \to A$ be an $A$-linear morphism, and $\mathsf{E}: A \to A$ be an idempotent $k$-algebra morphism such that $\mathsf{E} \circ \mathsf{D}^\circ =0$ and $\mathsf{D} \circ \mathsf{E} =0$. Suppose that the map $\mathsf{K}: A \to A$, defined as $\mathsf{K}= \mathsf{L} + \mathsf{E}$, is an isomorphism. Then $\xymatrix{A \ar@/^1.15pc/[r]^-{\tilde{\mathsf{L}}}  &  \ar@/^1.15pc/[l]^-{\mathsf{P}}  \mathsf{ker}(\mathsf{E}) 
  }$ is an FTC-pair, where $\tilde{\mathsf{L}}(a) = \mathsf{L}(a)$ and $\mathsf{P}(x) = \mathsf{K}^{-1}(x)$. 
\end{lemma}
\begin{proof} As explained above, $\mathsf{L}$ is a derivation and therefore $\tilde{\mathsf{L}}: A \to \mathsf{ker}(\mathsf{E})$ is clearly also a derivation. Before proving that $\mathsf{P}: \mathsf{ker}(\mathsf{E}) \to A$ is an integration, we will first show that FTC1 and FTC2 hold. 

Starting with FTC1, for an $x \in \mathsf{ker}(\mathsf{E})$, we easily compute that: 
\[ \tilde{\mathsf{L}}\left( \mathsf{P}(x) \right) = \mathsf{L}\left( \mathsf{K}^{-1}(x) \right) + \underbrace{\mathsf{E}(x)}_{=0} = \mathsf{K}\left( \mathsf{K}^{-1}(x) \right) = x \]
So $\tilde{\mathsf{L}} \circ \mathsf{P} = \mathsf{id}_{\mathsf{ker}(\mathsf{E})}$, and thus FTC1 holds. To show FTC2, first observe that since $\mathsf{D} \circ \mathsf{E} =0$, we have that $\mathsf{L} \circ \mathsf{E} =0$. Therefore, $\tilde{\mathsf{L}} \circ \mathsf{E} =0$ and so for an $a \in A$, we denote $c_a = \mathsf{E}(a) \in \mathsf{ker}(\tilde{\mathsf{L}})$. Since $\mathsf{E}$ is an algebra morphism, we also have that $c_a c_b = c_{ab}$. Moreover, since $\mathsf{E}$ is idempotent, $\mathsf{L} \circ \mathsf{E} =0$ also implies that $\mathsf{K} \circ \mathsf{E} = \mathsf{E}$. Then post-composing both sides by $\mathsf{K}^{-1}$ gives us that $\mathsf{K}^{-1} \circ \mathsf{E} = \mathsf{E}$. Then for all $a \in A$, we compute that: 
\[  \tilde{\mathsf{L}}\left( \mathsf{P}(a) \right) + c_a = \mathsf{K}^{-1}\left( \mathsf{L}(a) \right) + \mathsf{E}(a) =  \mathsf{K}^{-1}\left( \mathsf{L}(a) \right)+  \mathsf{K}^{-1}\left( \mathsf{E}(a) \right) =  \mathsf{K}^{-1}\left( \mathsf{K}(a)  \right) = a  \]
So we conclude that FTC2 holds as desired. 

Lastly, we need to show that $\mathsf{P}: \mathsf{ker}(\mathsf{E}) \to A$ is indeed an integration. To do so, first observe that for $x \in \mathsf{ker}(\mathsf{E})$ that $\mathsf{K}(x) = \mathsf{L}(x)$ and therefore, since $\mathsf{L}$ is a derivation, we have that the following equality holds for all $x,y \in \mathsf{ker}(\mathsf{E})$:
\[ \mathsf{K}(xy) = x \mathsf{K}(y) + y \mathsf{K}(x) \]
Then for $x,y \in \mathsf{ker}(\mathsf{E})$, we compute that: 
\begin{gather*}
\mathsf{P}(x) \mathsf{P}(y) = \mathsf{K}^{-1}(x) \mathsf{K}^{-1}(y) = \mathsf{K}^{-1}\left( \mathsf{K}\left( \mathsf{K}^{-1}(x) \mathsf{K}^{-1}(y) \right) \right)  = \mathsf{K}^{-1}\left( \mathsf{K}^{-1}(x) \mathsf{K}\left( \mathsf{K}^{-1}(y) \right) \right)  + \mathsf{K}^{-1}\left( \mathsf{K}^{-1}(y) \mathsf{K}\left( \mathsf{K}^{-1}(x) \right) \right)\\
= \mathsf{K}^{-1}\left( \mathsf{K}^{-1}(x) y \right)  + \mathsf{K}^{-1}\left( \mathsf{K}^{-1}(y) x \right) = \mathsf{P}\left( \mathsf{P}(x) y \right)  + \mathsf{P}\left( \mathsf{P}(y) x \right)
\end{gather*}
So $\mathsf{P}$ satisfies the Rota-Baxter rule and is therefore an integration. So we conclude that $\xymatrix{A \ar@/^1.15pc/[r]^-{\tilde{\mathsf{L}}}  &  \ar@/^1.15pc/[l]^-{\mathsf{P}}  \mathsf{ker}(\mathsf{E}) 
  }$ is an FTC-pair as desired. 
\end{proof}

\begin{example}\normalfont \label{ex:der-FTC} Let's apply Lemma \ref{lem:der-FTC} to the derivations from Ex \ref{ex:FTC-pairs}.(\ref{ex:FTC-poly}) and (\ref{ex:FTC-smooth}) to obtain new FTC-pairs. 
\begin{enumerate}[{\em (i)}]
\item Let $k$ be a field of characteristic zero and consider the polynomial ring $k[x]$ with the standard derivation ${\mathsf{D}: k[x] \to k[x]}$, so $\mathsf{D}(x^{n}) = n x^{n-1}$. Take $\mathsf{D}^\circ: k[x] \to k[x]$ to be the map which multiplies a polynomial by $x$, so $\mathsf{D}^\circ(p(x))= p(x)x$, and take $\mathsf{E}: k[x] \to k[x]$ to be the map which evaluates a polynomial at zero, $\mathsf{E}(p(x)) = p(0)$ -- which is clearly an idempotent algebra morphism. Moreover, since the derivative of a constant is zero, we have that $\mathsf{D} \circ \mathsf{E} =0$, and we also clearly have that $\mathsf{E} \circ \mathsf{D}^\circ =0$ as well. As a result, $\mathsf{L}: k[x] \to k[x]$ and ${\mathsf{K}: k[x] \to k[x]}$ are worked out to be given as follows: 
\begin{align*}
    \mathsf{L}(x^n) = n x^n && \mathsf{K}(x^n) = \begin{cases} 1 & \text{if $n=0$, so $x^0=1$} \\
  n x^n & \text{if $n\geq 1$}
    \end{cases}
\end{align*}
Then $\mathsf{K}$ is an isomorphism whose inverse is given as follows: 
\begin{align*}
  \mathsf{K}^{-1}(x^n) = \begin{cases} 1 & \text{if $n=0$, so $x^0=1$} \\
    \frac{1}{n} x^n & \text{if $n\geq 1$}
    \end{cases}
\end{align*}
While these maps $\mathsf{L}$, $\mathsf{K}$, and $\mathsf{K}^{-1}$ may seem somewhat unknown, they play a crucial role in the differential category with antiderivatives of polynomials \cite[Ex 7.3]{cockett_lemay_2018}. Lastly, observe that the kernel of $\mathsf{E}$ is given by reduced polynomials, that is, $\mathsf{ker}(\mathsf{E}) = k[x]_+ = \lbrace p(x) \in k[x] \vert~ p(0) = 0 \rbrace$. Therefore, by Lemma \ref{lem:der-FTC}, we obtain an FTC-pair $\xymatrix{ k[x] \ar@/^/[r]^-{\tilde{\mathsf{L}}}  &  \ar@/^/[l]^-{\mathsf{P}} k[x]_+}$ where concretely: 
\begin{align*}
    \tilde{\mathsf{L}}(x^n) = n x^n && \mathsf{P}(x^{n+1}) = \frac{1}{n+1} x^{n+1}
\end{align*}
Note that the $\tilde{\mathsf{L}}$-constants are precisely the constants, so $\mathsf{ker}\left( \tilde{\mathsf{L}} \right) \cong k$ and thus this FTC-pair is also augmented. 
\item Consider the $\mathbb{R}$-algebra of smooth functions $\mathcal{C}^\infty(\mathbb{R})$ with the standard derivation $\mathsf{D}: \mathcal{C}^\infty(\mathbb{R}) \to \mathcal{C}^\infty(\mathbb{R})$, so $\mathsf{D}(f)(x) = f^\prime(x)$. Similarly as above, define $\mathsf{D}^\circ: \mathcal{C}^\infty(\mathbb{R}) \to  \mathcal{C}^\infty(\mathbb{R})$ as $\mathsf{D}^\circ\left( f \right)(x) = f(x)x$, and define ${\mathsf{E}: \mathcal{C}^\infty(\mathbb{R}) \to \mathcal{C}^\infty(\mathbb{R})}$ as $\mathsf{E}(f)(x) = f(0)$. Again, we see that we have $\mathsf{E} \circ \mathsf{D}^\circ =0$ and $\mathsf{D} \circ \mathsf{E} =0$. So ${\mathsf{L}: \mathcal{C}^\infty(\mathbb{R}) \to \mathcal{C}^\infty(\mathbb{R})}$ and $\mathsf{K}: \mathcal{C}^\infty(\mathbb{R}) \to \mathcal{C}^\infty(\mathbb{R})$ are worked respectively to be:
\begin{align*}
    \mathsf{L}(f)(x) = f^\prime(x)x && \mathsf{K}(f)(x) = f^\prime(x)x + f(0)
\end{align*}
As was shown in \cite[Prop 6.1]{cruttwell2019integral}, $\mathsf{K}$ is an isomorphism with inverse given as follows:
\begin{align*}
  \mathsf{K}^{-1}(x^n) = \int \limits^1_0 \int^1_0 f(st x)x  \mathsf{d}s \mathsf{d}t + f(0)
\end{align*}
See \cite[Sec 6]{cruttwell2019integral} for an explanation on how $\mathsf{L}$, $\mathsf{K}$, and $\mathsf{K}^{-1}$ arise in calculus.  Lastly, observe that the kernel of $\mathsf{E}$ is given by reduced smooth functions, that is, $\mathsf{ker}(\mathsf{E}) = \mathcal{C}^\infty(\mathbb{R})_+ = \lbrace f \in \mathcal{C}^\infty(\mathbb{R}) \vert~ f(0) = 0 \rbrace$. Therefore, by Lemma \ref{lem:der-FTC}, we obtain an FTC-pair $\xymatrix{ \mathcal{C}^\infty(\mathbb{R}) \ar@/^/[r]^-{\tilde{\mathsf{L}}}  &  \ar@/^/[l]^-{\mathsf{P}} \mathcal{C}^\infty(\mathbb{R})_+}$ where concretely: 
\begin{align*}
    \tilde{\mathsf{L}}(f)(x) = f^\prime(x)x && \mathsf{P}(f)(x) = \int \limits^1_0 \int\limits^1_0 f(st x)x  \mathsf{d}s \mathsf{d}t
\end{align*}
Note that the $\tilde{\mathsf{L}}$-constants are precisely the constant functions, so $\mathsf{ker}\left( \tilde{\mathsf{L}} \right) \cong \mathbb{R}$ and thus this FTC-pair is also augmented. 
\end{enumerate}
\end{example}

We can also apply Lemma \ref{lem:der-FTC} when starting from a differential algebra $A$ with derivation $\mathsf{D}$. In fact, in this scenario, we no longer need the $A$-module morphism $\mathsf{D}^\circ$. To see this, consider the module of Kähler differential of $A$ \cite[Def 26.C]{matsumura1970commutative} $\Omega(A)$ with universal derivation $\mathsf{d}: A \to \Omega(A)$. By the universal property of the Kähler differentials \cite[Prop 26.1]{matsumura1970commutative}, there exists a unique $A$-module morphism $\mathsf{D}^\sharp: \Omega(A) \to A$ such that $\mathsf{D}^\sharp \circ \mathsf{d} = \mathsf{D}$. In other words, this says that $\mathsf{L} = \mathsf{D}$. Then when we also have our idempotent $k$-algebra morphism $\mathsf{E}: A \to A$, we can define our map ${\mathsf{K}: A \to A}$ which is now $\mathsf{K} = \mathsf{D} + \mathsf{E}$. So we can ask $\mathsf{K}$ to be an isomorphism, from which we can then obtain an FTC-pair. 

\begin{corollary} \label{cor:difallg-FTC} Let $A$ be a differential algebra with derivation $\mathsf{D}: A \to A$ and equipped with an idempotent $k$-algebra morphism $\mathsf{E}: A \to A$ such that $\mathsf{D} \circ \mathsf{E} = 0 = \mathsf{E} \circ \mathsf{D}$. If the map $\mathsf{K}: A \to A$ defined as $\mathsf{K} = \mathsf{D} + \mathsf{E}$ is an isomorphism, then $\xymatrix{A \ar@/^1.15pc/[r]^-{\tilde{\mathsf{D}}}  &  \ar@/^1.15pc/[l]^-{\mathsf{P}}  \mathsf{ker}(\mathsf{E}) 
  }$ is an FTC-pair, where $\tilde{\mathsf{D}}(a) = \mathsf{D}(a)$ and $\mathsf{P}(a) = \mathsf{K}^{-1}(a)$. 
\end{corollary}

\bibliographystyle{plain}      
\bibliography{FTCZINref}   

\begin{thebibliography}{10}

\bibitem{alvarez2022algebraic}
M.~Alvarez, R.~J{\'u}nior, and I.~Kaygorodov.
\newblock {The algebraic and geometric classification of Zinbiel algebras}.
\newblock {\em Journal of Pure and Applied Algebra}, 226(11):107106, 2022.

\bibitem{bagnol2016shuffle}
M.~Bagnol, R.~Blute, R.~Cockett, and J.-S.~P. Lemay.
\newblock The shuffle quasimonad and modules with differentiation and
  integration.
\newblock {\em Electronic Notes in Theoretical Computer Science}, 325:29--45,
  2016.

\bibitem{bai2013relative}
C.~Bai, L.~Guo, and X.~Ni.
\newblock {Relative Rota--Baxter operators and tridendriform algebras}.
\newblock {\em Journal of Algebra and Its Applications}, 12(07):1350027, 2013.

\bibitem{Blute2019}
R.F. Blute, J.R.B. Cockett, J-S.~P. Lemay, and R.A.G. Seely.
\newblock Differential categories revisited.
\newblock {\em Applied Categorical Structures}, 28:171--235, 2020.

\bibitem{cockett_lemay_2018}
J.R.B. Cockett and J-S.~P. Lemay.
\newblock Integral categories and calculus categories.
\newblock {\em Mathematical Structures in Computer Science}, pages 1--66, 2018.

\bibitem{cruttwell2019integral}
G.~S.~H. Cruttwell, J-S.~P. Lemay, and R.~B.~B. Lucyshyn-Wright.
\newblock {Integral and differential structure on the free $C^\infty$-ring
  modality}.
\newblock {\em Cahiers de topologie et g{\'e}om{\'e}trie diff{\'e}rentielle
  cat{\'e}goriques}, 62(2):116--176, 2021.

\bibitem{guo2012introduction}
L.~Guo.
\newblock {\em {An introduction to Rota-Baxter algebra}}.
\newblock International Press Somerville, 2012.

\bibitem{guo2008differential}
L.~Guo and W.~Keigher.
\newblock {On differential Rota--Baxter algebras}.
\newblock {\em Journal of Pure and Applied Algebra}, 212(3):522--540, 2008.

\bibitem{guo2014integro}
L.~Guo, G.~Regensburger, and M.~Rosenkranz.
\newblock On integro-differential algebras.
\newblock {\em Journal of Pure and Applied Algebra}, 218(3):456--473, 2014.

\bibitem{kaplansky1957introduction}
I.~Kaplansky.
\newblock {\em An introduction to differential algebra}.
\newblock Hermann, 1957.

\bibitem{lemay2020convenient}
J-S.~P. Lemay.
\newblock {Convenient antiderivatives for differential linear categories}.
\newblock {\em Mathematical Structures in Computer Science}, 30(5):545--569,
  2020.

\bibitem{loday1995cup}
J.-L. Loday.
\newblock {Cup-product for Leibniz cohomology and dual Leibniz algebras}.
\newblock {\em Mathematica Scandinavica}, pages 189--196, 1995.

\bibitem{Loday2001}
J.-L. Loday.
\newblock {\em {Dialgebras}}, pages 7--66.
\newblock Springer Berlin Heidelberg, 2001.

\bibitem{matsumura1970commutative}
H.~Matsumura.
\newblock {\em {Commutative Algebra}}, volume 120.
\newblock WA Benjamin New York, 1970.

\bibitem{uchino2008quantum}
K.~Uchino.
\newblock {Quantum analogy of Poisson geometry, related dendriform algebras and
  Rota--Baxter operators}.
\newblock {\em Letters in Mathematical Physics}, 85(2-3):91--109, 2008.

\end{thebibliography}
\end{document}